\documentclass[11pt]{article}
\usepackage[english]{babel}

\usepackage[utf8]{inputenc}

\usepackage{makeidx}         
\usepackage{graphicx}        
\usepackage{psfrag}          
\usepackage{multicol}        
\usepackage{amsmath}
\usepackage{amssymb}
\usepackage{amsfonts}
\usepackage{amsthm}
\usepackage{latexsym}
\usepackage{subfigure}                      
\usepackage{float}
\usepackage{mathrsfs}

\usepackage{graphicx}
\usepackage[hang]{caption2}
\usepackage{indentfirst}
\usepackage[section]{placeins} 
\usepackage{hyperref}
\usepackage{color}
\usepackage[normalem]{ulem}

\usepackage{latexsym}
\usepackage{amsmath}
\usepackage{amsfonts}
\usepackage{mathrsfs}
\usepackage{amstext}
\usepackage{amssymb}
\usepackage{amsthm} 
\usepackage{color}

\usepackage{tabu}
\usepackage{multirow}
\usepackage{array}

\usepackage{xspace}
\usepackage{verbatim}
\usepackage{comment}

\newtheorem{thm}{Theorem}[section]
\newtheorem{prop}[thm]{Proposition}
\newtheorem{cor}[thm]{Corollary}

\newtheorem{defn}[thm]{Definition}
\newtheorem{rem}[thm]{Remark}
\newtheorem{lem}[thm]{Lemma}

\def\XXint#1#2#3{{\setbox0=\hbox{$#1{#2#3}{\int}$} 
		\vcenter{\hbox{$#2#3$}}\kern-.5\wd0}}

\usepackage[margin=1.5cm]{geometry}

\def\e{\varepsilon}

\let\e=\varepsilon
\let\s=\sigma
\let\d=\delta

\def\R{{\mathbb{R}}}

\def\Z{{\mathbb{Z}}}

\def\Tes{{{\cal T}_{\varepsilon}^*}}
\def\Qed{{{\cal Q}_\e^{\diamond}}}

\def\Teb{\mathcal{T}^{b}_{\e}}

\let\O=\Omega

\def\gu{{\bf u}}
\def\gb{{\bf b}}

\def\gf{{\bf f}}
\def\gk{{\bf k}}

\usepackage[section]{placeins}
\numberwithin{equation}{section}     
\usepackage{graphics}
\usepackage{pdfsync}
\usepackage{hyperref}
\usepackage{color}
\usepackage[normalem]{ulem}

\let\O=\Omega

\graphicspath{{./pictures/}}

\def\Y{{\bf Y}}
\def\ds{\displaystyle}

\theoremstyle{plain}

\theoremstyle{definition}
\newtheorem{definition}{\bf Definition}

\def\e{\varepsilon}
\def\d{\delta}

\def\O{\Omega}
\def\o{\omega}

\def\R{{\mathbb{R}}}
\def\Z{{\mathbb{Z}}}

\def\Ga{{\bf a}}
\def\Gb{{\bf b}}

\def\Ge{{\bf e}}

\def\Gk{{\bf k}}

\def\GH{{\bf H}}

\def\GN{{\bf N}}

\def\GY{{\bf Y}}

\def\d{\delta}

\def\Tes{{\cal T}^*_\e}

\title{
 Homogenization of contact problem with Coulomb's friction on periodic cracks
}
\date{\ }

\author{\ G. Griso, 
	J. Orlik}

\begin{document}
	
	\maketitle

\noindent
\begin{abstract}
We consider the elasticity problem in a 
domain with 
contact on multiple periodic open cracks. 
The contact is described by the Signorini and
Coulomb-friction conditions.  Problem is non-linear, the dissipative
functional depends on the un-known solution and the existence of the
solution for fixed period of the structure is usually proven by the fix-point  argument in the Sobolev
spaces with a little higher regularity, $H^{1+\alpha}$. We rescaled norms, trace, jump and Korn inequalities in fractional Sobolev spaces with positive and negative exponent, using the unfolding technique, introduced by Griso,
Cioranescu and Damlamian. Then we proved the existence and uniqieness of the solution for friction and period fixed. Then we proved the continuous dependency of the solution to the problem with Coulomb's friction on the given friction and then estimated the solution using fixed point theorem.  However, we were not able to pass to the strong limit in the frictional dissipative term.  For this reason, we regularized the problem by adding a fourth-order term, which increased the regularity of the solution and allowed the passing to the limit. This can be interpreted as micro-polar elasticity.
\end{abstract}

\noindent
{\bf Keywords:}  \par
\vspace{9pt}

\noindent 
{{\bf 2010 Mathematics Subject Classification:}}  %
%


\section{Introduction}
 This paper deals with a static multi-scale contact problem with Coulomb's friction, which arise by time-discretization of a quasi-static problem, as it is shown in \cite{EJ00},\cite{EJ05}.  Such a problem results in a quasi-variational inequality, whose solvability was studied  in \cite{EJ00},\cite{EJ05}. In this paper, we want to repeat these results for multi-scale periodic 
	open cracks. \\
	We show that the Coulomb-friction problem admits solutions for every fixed period. 
	Then
	we obtain all compactness results and preliminary estimates in terms of the powers of the small parameter, related to the period of the structure, using Korn's inequalities and their rescaling.
	For this reason, we extend results of \cite{CDO} and \cite{GMO}, where asymptotic analysis and one-side Korn inequalities were given for multi-scale contact problem in a periodic domain with Tresca friction, and extend  the unfolding tools introduced by \cite{cdg1} and \cite{cddgz} to fractional order Sobolev spaces. We introduce tools for rescaling their norms and dual norms over domains and manifolds.\\
	We regularize the problem by adding the fourth-order term, in order to pass to the limit and prove the strong convergence of the interface stresses or co-normal derivatives on the cracks. We used the shifting technique as in \cite{EJ05}, or \cite{Radu} to show a better regularity in the macroscopic variable and, then, the strong convergence. Finally, we proved the existence and uniqueness of the solution to the limiting regularized problem.  It is easy to show that the solution of the  regularized $\e$-problem converges to the solution of the contact Coulomb's problem as the regularizing parameter tends to zero. However, it cannot be sent to zero after the passing to the limit, since it enters a denominator of the contruction condition for the fixed point argument. In the conclusion, we were able to homogenize just a regularized problem. The result is still new and interesting and can be applied to micro-polar materials with some additional rotational degrees of freedom.

\section{Geometric set up}
The problem is set in the natural space $\R^{3}$. Denote $\O$ a bounded domain in $\R^3$ with Lipschitz boundary and $\Gamma$ a  subset of $\partial \O$ with a positive Lebesgue measure on $\partial \O$. \\
In the following $\Y\doteq (0,1)^3$,  is the reference cell.   The  crack,  is a closed set denoted  $S$ and strictly included in $\Y$ and called ``open crack''. The crack $S$, is ``open"  in the sense that $\Y^*=\Y\setminus S$ lies on both sides of this surface. The set matrix is $\Y^*\doteq \Y\setminus  S$ (see Figure 1). \\
We assume that $S$ is the closure of an open connected set of the boundary of a domain ${\cal S}$ strictly included in $\Y$ and whose boundary is ${\cal C}^{1,1}$. We  denote  $\nu$ the outward unit normal vector to the boundary of the domain ${\cal S}$ (it belongs to $W^{1,\infty}(\partial {\cal S})^3$). 
%

Recall that in the periodic setting, almost every point  $z\in\R^3$  can be written as
$$z= \bigl[z\bigr] + \bigl\{z\bigr\},\quad [z]\in \Z^3,\quad \{z\}\in \Y.$$ 

Denote 
$$\Xi_\e=\big\{\xi \in \Z^3\;|\; \e \xi+\e \Y\subset \O\big\},\qquad \widehat\O_{\e}=\hbox{interior}\Big\{\bigcup_{\xi\in \Xi_\e}\big(\e \xi+\e\overline \Y\big)\Big\},\qquad \Lambda_{\e}=\O \setminus \overline{\widehat\O_{\e}},$$ this last set contains the parts from cells intersecting the boundary $\partial\O$.

\begin{center}
	\includegraphics[width=12cm]{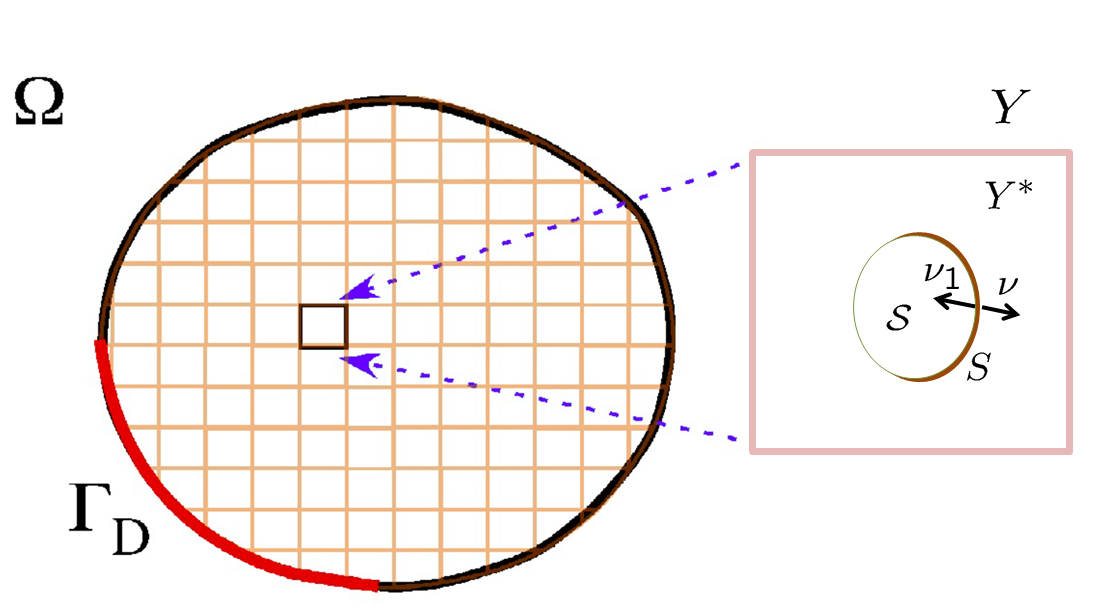}\\
	{Figure 1. Bounded domain with periodically distributed open and closed cracks}
\end{center}
The union of all the cracks is denoted $S_{\e}$, 
$$S_{\e} \doteq \Big\{x\in \,\widehat\O_{\e}\;\;\Big|\;\;\left\{\frac{x}{ \e}\right\} \in S \Big\}\subset {\cal S}_\e\doteq \Big\{x\in \,\widehat\O_{\e}\;\;\Big|\;\;\left\{\frac{x}{ \e}\right\} \in {\cal S} \Big\}$$
and the cracked domain
$$\O_{\e}^{*} = \O\setminus S_{\e}.$$
Set
$$
H^1_{\Gamma}(\O^*_\e)=\big\{ u\in H^1(\O^*_\e)\;|\; u=0\;\;\hbox{a.e. on } {\Gamma}\big\},\qquad H^2_{\Gamma}(\O^*_\e)=H^1_{\Gamma}(\O^*_\e)\cap H^2(\O^*_\e).
$$
\subsubsection*{Notations}
\textbullet  \; The normal component of a vector field $v $ on the boundary of a domain with Lipschitz boundary is denoted $v_{\nu}$, while the tangential component $v-v_{\nu}\nu$ is denoted $v_{\tau_\e}$ (where $\nu$ is the outward unit normal to the boundary),

\textbullet  \; the strain tensor of a vector field $v$ is denoted by $e(v)$; its values are symmetric 3x3 real matrices,

\textbullet  \; we use the notations of \cite{cdg1,cddgz} for the unfolding method.

\section{Preliminary results}

\subsection{ Recall on Poincar\'e, Poincar\'e-Wirtinger and Korn inequalities}
In the following,  for every  open bounded set  ${\cal O}\subset \R^3$ and  $\varphi\in L^{1}({\cal O})$,     ${\cal M}_{{\cal O}}(\varphi)$ denotes the mean value of $\varphi$  over   ${\cal O}$,  i.e.,  
$${\cal M}_{{\cal O}}(\varphi) = \frac{1}{ |{\cal O}|}\int_{{\cal O}} \varphi\, dy.$$ 

	Let ${\cal O}$ be a bounded domain in $\R^3$ with Lipschitz boundary.  In ${\cal O}$, the following Poincar\'e and  Poincar\'e-Wirtinger inequalities hold: 
\begin{equation*}
\forall \phi \in H^{1}({\cal O}),\qquad 
		\begin{aligned}
		&\|\phi\|_{H^{1}({\cal O})}\leq C \big(\|\nabla \phi\|_{L^2({\cal O})}+\|\phi\|_{L^2(\partial {\cal O})}\big),\\
		&\|\phi-{\cal M}_{{\cal O}}(\phi)\|_{L^2({\cal O})}\leq C \|\nabla \phi\|_{L^2({\cal O})}.	
		\end{aligned}
\end{equation*}
where the constant $C$ depends on ${\cal O}$. 

 Recall that the space of rigid displacements  is
$${\cal R}=\big\{ r \;|\; r(x)= a+b\land x,\;\; (a,b)\in \R^3\times \R^3\big\}.$$
A bounded domain ${\cal O}\subset \R^3$ satisfies the Korn-Wirtinger inequality if there exists a constant $C$ such that for every $v\in H^1({\cal O})^3$ there exists $r\in {\cal R}$ such that
\begin{equation}\label{Eq. 37}
\|v-r\|_{H^1({\cal O})}\le C\|e(v)\|_{L^2({\cal O})}.
\end{equation}
We equip $H^1({\cal O})^3$ with the  scalar product
\begin{equation}\label{PS}
<<u,v>>=\int_{\cal O} e(u):e(v)\, dx+\int_{\cal O} u\cdot v\, dx.
\end{equation}
If ${\cal O}\subset \R^3$ is  a bounded domain with Lipschitz boundary, ${\cal O}$ satisfies the Korn-Wirtinger inequality and the associated norm is equivalent to the usual norm of $H^1({\cal O})^3$. For the displacements in $H^1({\cal O})^3$ one also has the following Poincar\'e inequality:
\begin{equation}\label{TR}
\forall v \in H^1({\cal O})^3, \quad \|v\|_{H^{1/2}(\partial {\cal O})}\leq C \big(\|e(v)\|_{L^2({\cal O})} + \|v\|_{L^2({\cal O})}\big)
\end{equation}
where $C$ depends on ${\cal O}$.

\begin{definition}\label{def1} Let  ${\cal O}$ be a bounded domain in $\R^3$ with Lipschitz boundary, we denote 
	$$W({\cal O})\doteq\Big\{v\in H^1({\cal O})^3\;\;|\;\; \int_{\cal O} v(x)\cdot r(x) \,dx=0\;\;\hbox{for all } r\in {\cal R}\Big\}.$$
The space $W({\cal O})$ is the orthogonal of ${\cal R}$ in $H^1({\cal O})^3$ for the  scalar product \eqref{PS}.
\end{definition}
Recall that there exists a constant $C$ such that for every $v\in W({\cal O})$
\begin{equation}\label{Eq. 38}
\|v\|_{H^1({\cal O})}\le C \|e(v)\|_{L^2({\cal O})}.
\end{equation}

\subsection{ The spaces $H^\alpha(S)$ and $H^\alpha(S_{\e})$, $\alpha\in (0,1)$}

The space $H^\alpha(S)$, $\alpha\in (0,1)$, is the following subspace of $L^2(S)$:
$$H^\alpha(S)\doteq \Big\{v\in L^2(S)\; | \; \int_{S}\int_{S}\frac{|v(x)-v(y)|^2}{|x-y|^{2+2\alpha}}d\sigma_x d\sigma_y<+\infty\Big\}.$$
It is equipped with the semi-norm
\begin{equation}
\|v\|^{'2}_{H^\alpha(S)}=\int_{S}\int_{S}\frac{|v(x)-v(y)|^2}{|x-y|^{2+2\alpha}}d\sigma_x d\sigma_y
\label{Hs-norm}
\end{equation} and the Sobolev-Slobodetsky norm
\begin{equation*}
\|v\|_{H^\alpha(S)}=\sqrt{\|v\|^2_{L^2(S)}+\|v\|^{'2}_{H^\alpha(S)}}.
\end{equation*}	
The space $H^\alpha(S_{\e})$, $\alpha\in (0,1)$ is the subspace of $L^2(S_{\e})$ containing the functions whose restrictions to every connected component of $S_{\e}$ belong to the space $H^\alpha$ of this connected component
$$
H^\alpha(S_{\e})\doteq \Big\{v\in L^2(S_{\e})\; | \; \sum_{\xi\in \Xi_\e}\int_{\e\xi+\e S}\int_{\e\xi+\e S}\frac{|v(x)-v(y)|^2}{|x-y|^{2+2\alpha}}d\sigma_x d\sigma_y<+\infty\Big\}.
$$ It is equipped with the semi-norm 
\begin{equation*}
\|v\|^{'2}_{H^\alpha(S_{\e})}=\sum_{\xi\in \Xi_\e}\int_{\e\xi+\e S}\int_{\e\xi+\e S}\frac{|v(x)-v(y)|^2}{|x-y|^{2+2\alpha}}d\sigma_x d\sigma_y
\end{equation*} and the Sobolev-Slobodetsky norm
\begin{equation}\label{NH1/2}
\|v\|_{H^\alpha(S_{\e})}=\sqrt{\|v\|^2_{L^2(S_\e)}+\e^{2\alpha}\|v\|^{'2}_{H^\alpha(S_{\e})}}.
\end{equation}

\subsection{Definition of the jumps}\label{Ss2.2}

For every $v\in H^1(\Y^*)$ denote $v_{\partial {\cal S}^-}$ the trace of $v_{|{\cal S}}$ on $\partial {\cal S}$ and $v_{\partial {\cal S}^+}$  the trace of $v_{|\Y^* \setminus \overline{\cal S}}$ on $\partial {\cal S}$. The jump of $v$ across the surface $S$ is 
$$[v]_{S}=\big(v_{|\partial {\cal S}^+}-v_{|\partial {\cal S}^-}\big)_{|S},\qquad [v]_{S}\in \widetilde{H}^{1/2}(S)\footnotemark.
\footnotetext{ $\widetilde{H}^{1/2}(S)$ is the space of all $v\in H^{1/2}(S)$ whose extension by $0$ in $\partial {\cal S}\setminus S$ belongs to $H^{1/2}(\partial {\cal S})$ (see \cite{Grisward}).} 
$$
We also set for every $v\in H^1(\Y^*)^3$
$$
 [v_\nu]_{S}= [ v ]_{S}\cdot\nu\in \widetilde{H}^{1/2}(S), \qquad  [v_\tau]_{S}= [v]_{S}- [v_\nu]_{S}\,\nu\in \widetilde{H}^{1/2}(S)^3.
 $$
In a similar way, for every $v\in H^1(\O^*_\e)^3$ one defines $[v]_{S_\e}$ and  $[v_{\nu_\e}]_{S_\e}$ and $ [v_{\tau_\e}]_{S_\e}$.

%

\subsection{Some recalls on the main periodic unfolding operators}

We follow the notations and definitions of \cite{cddgz}.
\begin{prop}\label{prop4.1} Assume $p \in [1,+\infty]$, then 
	\begin{itemize}
	\item  for every  $\phi \in L^{p}(\O^*_\e)$  one has (recall that $|\Y|=1$)
$$\|\Tes(\phi)\|_{L^{p}(\O\times \Y^{\ast})}=\|\phi\|_{L^{p}(\widehat\O^*_\e)}\leq \|\phi\|_{L^{p}(\O^*_\e)}.$$
		\item for every	$\phi \in W^{1,p}(\O^*_\e)$ one has $\nabla_{y}\Tes(\phi)(x,y)=\e \Tes(\nabla \phi)(x,y)$ for a.e. $(x,y)\in \O\times\Y^*$ and 
		$$\|\nabla_{y}\Tes(\phi)\|_{L^{p}(\O\times \Y^*)}=\e\|\nabla \phi\|_{L^{p}(\widehat\O^*_\e)}\leq \e\|\nabla \phi\|_{L^{p}(\O^*_\e)}.$$
	\end{itemize}
\end{prop}

\begin{defn}   Assume $p \in [1,+\infty]$. The operator  $\mathcal{T}^b_\e$ from $L^p(S_{\e})$ into $L^p(\O\times S)$ is defined by 
	\[
\forall \phi \in L^p(S_{\e}),\qquad 
	\left\{\begin{aligned}
	\mathcal{T}^b_\e(\phi)(x,y)&=\phi(\e[x/\e]+\e y)\, \quad \hbox{for a.e. } (x,y)\in \widehat{\O}_\e\times S,\\
	\mathcal{T}^b_\e(\phi)(x,y)&=0\, \hskip 25mm  \hbox{for a.e. } (x,y)\in \Lambda_{\e}\times S.
	\end{aligned}\right.
	\]
\end{defn}

\begin{prop}\label{prop3.3} Assume $p \in [1,+\infty]$, then  for every $\phi \in L^{p}(S_{\e})$ one has
	\[
	\begin{aligned}
	&\int_{\O \times S}\mathcal{T}^b_\e(\phi)(x,y)\, dxd\sigma_y=\e\int_{S_{\e}}\phi(x)\, d\sigma_x,\\
	&\|\mathcal{T}^b_\e(\phi)\|_{L^{p}(\O \times S)}=\e^{1/p}\|\phi\|_{L^p(S_{\e})}.
	\end{aligned}
	\]
\end{prop}

\subsection{Estimates in $L^2(\O; H^{\alpha}(S))$, 
$\alpha\in (0,1)$.}

As immediate consequence of the definitions of the  semi-norms in $H^{\alpha}(S)$ and $H^{\alpha}(S_\e)$, we have  the following lemma and corollary:
\begin{lem}\label{lem46} For every $u\in H^{\alpha}(S_\e)$, $\alpha\in(0,1)$, one has
	$$\|\mathcal{T}^b_\e(u)\|'_{L^2(\O ; H^\alpha(S))}=\e^{1/2+\alpha}\|u\|'_{H^\alpha(S_\e)}.$$
\end{lem}
\begin{proof} One has 
	\begin{align*}
	\big(\|u\|'_{H^\alpha(S_\e)}\big)^2&=\sum_{\xi\in \Xi_\e}\int_{\e\xi+\e S}\int_{\e\xi+\e S}\frac{|u(x)-u(y)|^2}{|x-y|^{2+2\alpha}}\, d\sigma_x d\sigma_y\\
	&= \sum_{\xi\in \Xi_\e}\int_{S}\int_{S}\e^4\frac{|u(\varepsilon\xi +\varepsilon s)-u(\varepsilon\xi +\varepsilon t)|^2}{\varepsilon^{2+2\alpha}|\xi +s-\xi - t|^{2+2\alpha}}\, d\sigma_s d\sigma_t \\
	&= \varepsilon^{-1-2\alpha}\int_{\O}\int_{S}\int_{S}\frac{|\mathcal{T}^b_\varepsilon (u)(x,s)-\mathcal{T}^b_\varepsilon (u)(x,t)|^2}{|s- t|^{2+2\alpha}}\, d\sigma_s d\sigma_t \, dx\\&= \varepsilon^{-1-2\alpha}\Big(\|\mathcal{T}^b_\varepsilon (u)\|^{'}_{L^2(\O;H^{\alpha}(S))}\Big)^2\, .
	\end{align*}
	The equality is proved.
\end{proof}
As immediate consequence of Propositions \ref{prop3.3} and \ref{lem46} one has 
\begin{cor}\label{unfoldet_trace} For every  $\phi\in {H}^\alpha(S_\e)$, $\alpha\in(0,1)$, one has
$$\|[\mathcal{T}^b_\e(\phi)]\|_{L^2(\O ; {H}^{\alpha}(S))}=\e^{1/2}\|[\phi]\|_{{H}^{\alpha}(S_\e)}.$$
For every $u\in H^1(\O^*_\e)^3$, one has
\begin{equation}\label{JumpS}
\begin{aligned}
\|[\mathcal{T}^b_\e(u)]_{S}\|_{L^2(\O ; {H}^{1/2}(S))}=\e^{1/2}\|[u]_{S_\e}\|_{{H}^{1/2}(S_\e)}\leq C_0 \e \|e(u)\|_{L^2(\O^*_\e)}.
\end{aligned}
\end{equation}
The constant does not depend on $\e$, it depends on the crack $S$.
\end{cor}
\begin{proof} Since the domain  $\Y^{*}$  satisfies   the Korn inequality, for every  $v\in H^1(\Y^{*} )^3$ there exists a rigid displacement $r$  such that
$$\|v-r \|_{H^{1}(\Y^* )}\leq C \|e_y(u)\|_{L^2(\Y^* )}.$$
The above inequality yields
\begin{equation}\label{JumpS-bis}
\|[v]_{S}\|_{L^2(S)}+\|[v]_{S}\|'_{H^{1/2}(S)}\leq C \|v-r\|_{H^{1}(\Y^* )}\leq C_0\|e_y(v)\|_{L^2(\Y^* )}.
\end{equation}
Applying the above estimate with $v=\Tes(u)(\e\xi,\cdot)$, $\xi\in \Xi_\e$, easily  yields \eqref{JumpS}.
\end{proof}

\subsection{Estimates for the displacements in $H^1_\Gamma(\O^*_\e)$,  their traces and jumps}\label{SS4.6}
Let $u$ be in $H^1_{\Gamma}(\O^*_\e)^3$ and $r_u(\e\xi,\cdot)$  the orthogonal projections of $u_{|\e\xi+\e {\Y}}$, $\xi\in \Xi_\e$, on ${\cal R}$, 
$$
r _{u}(\e\xi,x)=a_u(\e\xi)+  b_u(\e\xi)\land \Big( \e\Big\{{x\over \e}\Big\}\Big),\qquad x \in \e\xi+\e {\Y}, \quad a_u(\e\xi),\; b_u(\e\xi)\in \R^3.
$$ 
We start with the Korn inequality for the cracked domain. 
\begin{lem} \label{prop5.1}
	There exists a constant $C$ (independent of $\e$) such that  for every $u\in H^1(\O_\e)^3$
	\begin{equation}
	\label{smalldomain}
	\begin{aligned}
	&\sum_{\xi\in \Xi_\e}\big(\|u-r _u \|^2_{L^{2}(\e\xi+\e{\Y}^*)}+ \e^2 \|\nabla (u-r _u)\|^2_{L^{2} (\e\xi+\e{\Y}^*)}\big) \leq  C \e^2 \|e(u)\|^2_{L^2(\O^*_{\e})},\\
	&\|u-r _u \|_{L^{2}(S_{\e})} \leq  C \sqrt \e \|e(u)\|_{L^2(\O^*_{\e})},\qquad 
	\|u-r _u \|'_{H^{1/2}(S_{\e})} \leq  C \|e(u)\|_{L^2(\O^*_{\e})}.
\end{aligned}
	\end{equation}
\end{lem}

\begin{proof} Applying \eqref{Eq. 38} (after $\e$-scaling) gives
	\begin{equation*}
	\|u-r _u\|^2_{L^{2}(\e\xi+\e {\Y}^*)}+ \e^2 \|\nabla (u-r _u)\|^2_{L^{2} (\e\xi+\e {\Y}^*)} \leq  C \e^2 \|e(u)\|^2_{L^2(\e\xi+\e {\Y})}.
	\end{equation*}
	Then adding the above inequalities  (with respect to $\xi$) yields  \eqref{smalldomain}$_1$. Then one obtains \eqref{smalldomain}$_{2,3}$.
\end{proof}

Next proposition provides the Korn inequality in terms of jumps and the estimate for jumps instead of traces.

\begin{lem}\label{prunk} For every displacement  $u$ in $H^{1}_\Gamma(\Omega^*_\e)^3$, one has
		\begin{equation}
		\label{eq315}
		\begin{aligned}
		&\|u\|_{H^1(\O^*_\e)}\leq C \| e(u) \|_{L^2 (\O^*_{\e})},\\
		&\|[u]_{S_\e}\|_{H^{1/2}(S_\e)}\leq C_0 \sqrt\e \|e(u)\|_{L^2(\O^*_\e)}.
		\end{aligned}
		\end{equation}	
	The constant does not depend on $\e$.
\end{lem}
\begin{proof} Proceeding as in  \cite{GMO}, we obtain the existence of a constant independent of $\e$ such that \eqref{eq315}$_{1}$ is satisfied. Then \eqref{eq315}$_{2}$  is given by \eqref{JumpS}.
\end{proof}

\section{The  unfolding operator from $H^{-\alpha}(S_{\e})$ into $L^2(\O; H^{-\alpha}(S))$, $\alpha\in (0,1)$}
\begin{definition} For $p\in[1,+\infty]$, the averaging boundary operator ${\cal U}_\e^b\;:\; L^p(\O\times S)\longmapsto L^p(S_\e)$ is defined by
	$$
\forall\Phi\in L^p(\O\times S),\qquad 	{\cal U}^b_\e(\Phi)(x)=\left\{
	\begin{aligned}
	&\int_Y \Phi\Big(\e\Big[{x\over \e}\Big]+\e z, \Big\{{x\over \e}\Big\}\Big)dz\qquad \hbox{for a.e. } x\in \widehat{\O}_\e,\\
	&0\hskip 50mm \hbox{for a.e. } x\in \Lambda_\e.
	\end{aligned}\right.
	$$
\end{definition}
\subsubsection{Some properties involving ${\cal U}^b_\e$ and ${\cal T}^b_\e$}
Let $p$ be in $[1,+\infty]$,  $p'$ its conjugate ($1/p+1/p'=1$)
\begin{itemize}
	\item $\forall\Phi\in L^p(\O\times S),\;\; \forall \psi\in L^{p'}(S_\e)$, $\displaystyle \int_{S_\e} \psi\,{\cal U}^b_\e(\Phi)\,d\sigma={1\over \e}\int_{\O\times S}{\cal T}^b_\e(\psi)\Phi\,dx d\sigma_y$,
	\item ${\cal U}^b_\e$ is almost a right inverse of ${\cal T}^b_\e$
	\begin{equation}\label{Prop1}
	\forall\Phi\in L^p(\O\times S),\quad {\cal T}^b_\e\circ{\cal U}^b_\e(\Phi)(x,y)=\left\{
	\begin{aligned}
	&\int_Y \Phi\Big(\e\Big[{x\over \e}\Big]+\e z, y\Big)dz\qquad \hbox{for a.e. for } (x,y)\in \widehat{\O}_\e\times S,\\
	&0\hskip 43mm \hbox{for a.e.  } (x,y)\in \Lambda_\e\times S.
	\end{aligned}\right.
	\end{equation}
	\item ${\cal U}^b_\e$ is a left inverse of ${\cal T}^b_\e$
	\begin{equation}\label{Prop2}
	\forall\phi\in L^p(S_\e),\qquad {\cal U}^b_\e\circ{\cal T}^b_\e(\phi)=\phi,
	\end{equation}
	\item for every $\Phi$ in $L^2(\O; H^\alpha(S))$ one has ${\cal U}^b_\e(\Phi)\in H^\alpha(S_\e)$.
\end{itemize}
For every $\alpha\in (0,1)$,  denote $H^{-\alpha}( S_\e)$ (resp. $\widetilde{H}^{-\alpha}(S_\e)$) the dual space of $H^{\alpha}(S_\e)$ (resp. $\widetilde{H}^{\alpha}(S_\e)$) equipped with the dual norm.
\vskip 2mm

Now, for $g\in H^{-\alpha}( S_\e)$ (resp. $\widetilde{H}^{-\alpha}(S_\e)$), $\alpha\in (0,1)$, one defines by duality ${\cal T}^b_\e(g)\in L^2(\O; H^{-\alpha}(S ))$ (resp. ${\cal T}^b_\e(g)\in L^2(\O; \widetilde{H}^{-\alpha}(S))$) as
$$
\begin{aligned}
\big\langle{\cal T}^b_\e(g), \Phi\big\rangle_{L^2(\O ; H^{-\alpha}(S)),L^2(\O ; H^{\alpha}(S))} &\doteq\e \big\langle g, {\cal U}^b_\e(\Phi)\big\rangle_{H^{-\alpha}(S_\e), H^\alpha( S_\e)}\qquad \forall \Phi\in L^2(\O; H^{\alpha}(S)),\\
\hbox{(resp.}\;\; \big\langle{\cal T}^b_\e(g), \Phi\big\rangle_{L^2(\O ; \widetilde{H}^{-\alpha}(S)),L^2(\O ; \widetilde{H}^{\alpha}(S))}& \doteq\e \big\langle g, {\cal U}^b_\e(\Phi)\big\rangle_{\widetilde{H}^{-\alpha}(S_\e), \widetilde{H}^{\alpha}(S_\e)}\qquad \forall \Phi\in L^2(\O; \widetilde{H}^{\alpha}(S))\hbox{)}
\end{aligned}
$$ Observe that
\begin{equation}\label{EQ44}
\begin{aligned}
&\big\langle{\cal T}^b_\e(g), {\cal T}^b_\e(\phi)\big\rangle_{L^2(\O ; H^{-\alpha}(S)),L^2(\O ; H^{\alpha}(S))} = \e \big\langle g, \phi\big\rangle_{H^{-\alpha}(S_\e), H^\alpha(S_\e)}\hskip 7mm \forall \phi\in H^\alpha(S_\e),\\
\hbox{(resp.}\;\; &\big\langle{\cal T}^b_\e(g), {\cal T}^b_\e(\phi)\big\rangle_{L^2(\O ; \widetilde{H}^{-\alpha}(S)),L^2(\O ; \widetilde{H}^{\alpha}(S))} = \e \big\langle g, \phi\big\rangle_{\widetilde{H}^{-\alpha}(S_\e), \widetilde{H}^\alpha(S_\e)}\qquad \forall \phi\in \widetilde{H}^\alpha(S_\e)\hbox{)}
\end{aligned}
\end{equation} since here
$$
\begin{aligned}
&\forall \phi\in H^\alpha(S_\e),\qquad {\cal U}^b_\e\circ \Tes(\phi)=\phi,\\
\hbox{(resp.}\;\; &\forall \phi\in \widetilde{H}^\alpha(S_\e),\qquad {\cal U}^b_\e\circ \Tes(\phi)=\phi\hbox{)}
\end{aligned}
$$

\begin{lem}\label{lem48} For every   $\alpha\in (0,1)$, one has
\begin{equation}\label{EQ35}
\begin{aligned}
\|{\cal T}^b_\e(\phi)\|_{L^2(\O; H^{-\alpha}(S))}=\e^{1/2}\|\phi\|_{H^{-\alpha}(S_\e)},\qquad \forall \phi\in H^{-\alpha}(S_\e),\\
\|{\cal T}^b_\e(\phi)\|_{L^2(\O; \widetilde{H}^{-\alpha}(S))}=\e^{1/2}\|\phi\|_{\widetilde{H}^{-\alpha}(S_\e)},\qquad \forall \phi\in \widetilde{H}^{-\alpha}(S_\e).
\end{aligned}
\end{equation}
\end{lem}
\begin{proof} One proves \eqref{EQ35}$_1$, the proof of \eqref{EQ35}$_2$ is obtained following the same lines. One has
	$$\forall \Phi\in L^2(\O;  H^{\alpha}(S)),\qquad  \big\langle{\cal T}^b_\e(g) , \Phi\big\rangle_{L^2(\O ; H^{-\alpha}(S)),L^2(\O ; H^{\alpha}(S))} \leq \e \|g\|_{H^{-\alpha}(S_\e)}\|{\cal U}^b_\e(\Phi)\|_{H^\alpha(S_\e)}.$$ Then, Property \eqref{Prop1} and Propositions \ref{prop3.3}-\ref{lem46} lead to
	$$
	\begin{aligned}
	&\|{\cal U}^b_\e(\Phi)\|_{L^2(S_\e)}=\e^{-1/2}\|{\cal T}^b_\e\circ {\cal U}^b_\e(\Phi)\|_{L^2(\O\times  )}\leq \e^{-1/2}\|\Phi\|_{L^2(\O\times S )},\\
	\e^{\alpha}&\|{\cal U}^b_\e(\Phi)\|'_{H^\alpha(S_\e)}=\e^{-1/2}\|{\cal T}^b_\e\circ {\cal U}^b_\e(\Phi)\|'_{L^2(\O; H^\alpha(S))}\leq \e^{-1/2}\|\Phi\|'_{L^2(\O; H^\alpha(S))}.
	\end{aligned}
	$$ Hence 
$$
\forall \Phi\in L^2(\O;  H^{\alpha}(S)),\qquad  \big\langle{\cal T}^b_\e(g) , \Phi\big\rangle_{L^2(\O ; H^{-\alpha}(S)),L^2(\O ; H^{\alpha}(S))} \leq \e^{1/2} \|g\|_{H^{-\alpha}(S_\e)}\|\Phi\|_{L^2(\O; H^\alpha(S))}.
$$ which yields $\|{\cal T}^b_\e(g)\|_{L^2(\O; H^{-\alpha}(S))}\leq \e^{1/2}\|g\|_{H^{-\alpha}(S_\e)}$. 
\smallskip
	
	Now, due to the Property \eqref{Prop2}, for every $\phi\in H^\alpha(S_\e)$ one has
	$$\big\langle g, \phi\big\rangle_{H^{-\alpha}(S_\e), H^\alpha(S_\e)}={1\over \e}\big\langle{\cal T}^b_\e(g), {\cal T}^b_\e(\phi)\big\rangle_{L^2(\O ; H^{-\alpha}(S)),L^2(\O ; H^{\alpha}(S))}.$$ So
	$$\big\langle g, \phi\big\rangle_{H^{-\alpha}(S_\e), H^\alpha(S_\e)}\leq {1\over \e}\|{\cal T}^b_\e(g)\|_{L^2(\O; H^{-\alpha}(S))}\| {\cal T}^b_\e(\phi)\|_{L^2(\O; H^{\alpha}(S))}.$$ The estimates in Propositions \ref{prop3.3}-\ref{lem46} yield $\| {\cal T}^b_\e(\phi)\|_{L^2(\O; H^{\alpha}(S))}\leq \e^{1/2}\|\phi\|_{H^\alpha(S_\e)}$, hence 
$$\|g\|_{H^{-\alpha}(S_\e)}\leq \e^{-1/2}\|{\cal T}^b_\e(g)\|_{L^2(\O; H^{-\alpha}(S))}.$$
The lemma is proved.
\end{proof}

\begin{lem}\label{lem-div}
Let  $h_\e$ be in $L^2({\cal S}_\e)$ and $v_\e$ be a field in $L^2({\cal S}_\e)^3$ such that
\begin{equation}
\hbox{div}\,(v_\e)=h_\e \qquad \hbox{in}\;\quad {\cal D}'({\cal S}_\e).
\end{equation}
One has
\begin{equation}\label{EstSigma}
\|v_\e\cdot \nu_\e\|_{H^{-1/2}(\partial {\cal S}_\e)}\leq C\big(\e\|h_\e\|_{L^2({\cal S}_\e)}+\|v_\e\|_{L^2({\cal S}_\e)}\big).
\end{equation}
The constant does not depend on $\e$.
\end{lem}
\begin{proof} Due to the hypothesis of the lemma, one has for every $\xi\in \Xi_\e$
$${\cal T}^b_\e(v_\e)(\e\xi,\cdot)\cdot\nu\in H^{-1/2}(\partial {\cal S}),\qquad \hbox{div}_y\,(\Tes(v_\e))=\e \Tes(h_\e) \qquad \hbox{in}\;\quad {\cal D}'(\O\times {\cal S})$$ and the estimate
$$\|{\cal T}^b_\e(v_\e)(\e\xi,\cdot)\cdot\nu\|^2_{H^{-1/2}(\partial{\cal S})}\leq C\big(\e^2\|{\cal T}_\e(h_\e)(\e\xi,\cdot)\|^2_{L^2( {\cal S})}+\|{\cal T}_\e(v_\e)(\e\xi,\cdot)\|^2_{L^2( {\cal S})}\big).$$
Adding  the above inequalities for $\xi\in\Xi_\e$ yields
$$ \|{\cal T}^b_\e(v_\e)\cdot\nu\|^2_{L^2(\O;H^{-1/2}(\partial{\cal S}))}\leq C\big(\e^2\|\Tes(h_\e)\|^2_{L^2(\O\times {\cal S})}+\|\Tes(v_\e)\|^2_{L^2(\O\times {\cal S})}\big)$$
which in turn with  Proposition \ref{prop4.1} give  \eqref{EstSigma}.
\end{proof}

\section{Statement of the contact $\e$-problem with Coulomb's friction on periodic cracks}

Assume that one has a given symmetric bilinear form on $H^1_{\Gamma}(\O^*_\e)^3$
$$
{\mathbf a}^{\e}({u},{v})\doteq \int_{\O^*_\e } a^{\e}_{ijkl}(x) e_{ij}(u)(x)\,e_{kl}(v)(x)\,dx,\qquad e_{ij}(u)={1\over 2}\Big({\partial u_i\over \partial x_j}+{\partial u_j\over \partial x_j}\Big),\quad (i,j)\in \{1,2,3\}^2
$$
where the tensor field $a^{\e}=(a^{\e}_{ijkl})$, $a^{\e}_{ijkl}\in L^\infty(\O^*_\e)$,  has the usual
properties of symmetry, boundedness  and coercivity  when operating on symmetric $3\times 3$ matrices
$$ 
a^{\e}_{ijkl}=a^{\e}_{jikl}=a^{\e}_{ klij},
\qquad \overline \alpha \;\eta_{ij}\eta_{ij}\leq a^{\e}_{ijkl}\, \eta_{ij} \eta_{kl}\leq C_A\, \eta_{ij} \eta_{kl}\qquad \hbox{a.e. in }\in \O^*_\e.
$$
The vector fields $v$ are the admissible displacement fields with respect to the
reference configuration $\O^*_{\e}$. The tensor field 
$$
\sigma^{\e}_{kl}(v)\doteq a^{\e}_{ijkl}e_{ij}(v)\qquad \forall v\in H^1(\O^*_\e)^3
$$
is the stress tensor associated to the strain tensor $e(v)$.
\smallskip

\noindent{\bf Assumptions} \label{Assum} 
	\begin{itemize}
		\item The functions $a_{ijkl}$ belong to $L^{\infty}(\Y^*)$ and they are $W^{1,\infty}$ in a neighborhood of $S$, $\ds a^\e_{ijkl}=a_{ijkl}\Big(\Big\{{\cdot\over \e}\Big\}\Big)$  a.e.  in $\R^3$,
		\item the applied volume forces $f$ belong to $L^2(\O)^3$, 
		\item the friction coefficient $\mu$ is a non-negative function belonging  to  $W^{1,\infty}(\O)$ with support included in $\O'\subset \O$,\footnote{ with a few change, we can choose a friction coefficient in the form $\ds \mu_\e(x)=\mu\Big(x,\Big\{{x\over \e}\Big\}\Big)$ for a.e. $x\in \O'\times S$ where $\mu\in W^{1,\infty}(\O\times S)$ vanishes in $\big(\O\setminus \overline{\O'}\big)\times S$).		
		}
		\item ${\cal K}_{\e}$ is a convex set  defined by
\begin{equation*}
{\cal K}_{\e}\doteq \big \{ v\in H^1_{\Gamma}(\O^*_\e )^3 \;\;|\;\; [v_{\nu_\e}]_{S_\e}\leq 0 \big\}.
\end{equation*} 
	\end{itemize}
The vector fields $v\in {\cal K}_\e$ are the admissible displacement fields with respect to the reference configuration $\O^*_{\e}$. The inequality in the definition of
${\cal K}_{\e}$ represent the non-penetration condition.\\

The strong formulation of the static contact problem is the following:
\begin{equation}
\left\{
\begin{aligned}
&\hbox{Find $u_{\e}$ in ${\cal K}_{\e}$ such that},\\
-&\text{div}\, \sigma_{\e}({u}_{\e}) = f\quad \text {in } \O^{*}_{\e},\\
&\sigma_{\nu_\e}({u}_{\e})[({u}_{\e})_{\nu_\e}]_{S_{\e}}=0, \quad 
\sigma_{\nu_\e}({u}_{\e})|_{S_{\e}}\leq 0, \quad  [\sigma_{\nu_\e}({u}_{\e})]_{S_{\e}}=0, \\
&|\sigma_{\tau_\e}({u}_{\e})|< \mu  |\s_{\nu_\e}({u}_{\e})|\quad \Rightarrow \quad
[({u}_{\e})_{\tau_\e}]_{S_{\e}}=0\quad \textrm{on}\quad S_{\e},\\
&\sigma_{\tau_\e}({u}_{\e})=-\mu  |\sigma_{\nu_\e}({u}_{\e})|
\frac{[({u}_{\e})_{\tau_\e}]_{S_{\e}}}{|[({u}_{\e})_{\tau_\e}]|_{S_{\e}}}\;\;
\Rightarrow \;\;[({u}_{\e})_{\tau_\e}]_{S_{\e}}\neq 0\quad \textrm{on}\quad S_{\e},\\
&\sigma_{\e}({u}_{\e})\cdot \nu =0 \quad\text{on}\quad \partial\O\setminus \Gamma,\\
& u_\e=0\quad \hbox{on}\quad \Gamma,
\end{aligned}
\right.
\label{prob_str}
\end{equation}
where $\sigma_{\nu_\e}(v)=(\sigma_{\e}(v)\, {\nu_\e}) \cdot {\nu_\e}$, $\sigma_{\tau_\e}(v)=\sigma_{\e}(v){\nu_\e}-\sigma_{\nu_\e}(v)\,{\nu_\e}$.\\

The weak formulation of the static contact problem is
\begin{equation} \label{var_stat}
\begin{aligned}
&\hbox{Find $U_{\e}\in {\cal K}_{\e}$ such that   for every $v\in {\cal K}_{\e}$},\\ 
&{\mathbf a}^{\e}(U_{\e},v-U_{\e})+  \big\langle\mu  \,|\sigma_{\nu_{\e}}(U_{\e})| , |[v_{\tau_{\e}}]_{S_{\e}}|-|[(U_\e)_{\tau_{\e}}]_{S_{\e}}|\big\rangle_\e\;  \geq \left( f, v - U_{\e} \right) 
\end{aligned}
\end{equation} where
$$
\begin{aligned}
&\forall (g,w)\in \widetilde{H}^{-1/2}(S_\e)\times \widetilde{H}^{1/2}(S_\e),\qquad \big\langle g\, , \, w\big\rangle_\e=\big\langle g\, , \, w\big\rangle_{\widetilde{H}^{-1/2}(S_\e) , \widetilde{H}^{1/2}(S_\e)},\\
&\forall v\in L^2(\O^*_\e)^3,\qquad ( f ,v )=\int_{\O^*_\e} f\cdot v\, dx.
\end{aligned}
$$ 
To solve the above problem, we consider the contact problem with given friction
\begin{equation}\label{var_stat_G-0}
\left\{\begin{aligned} 
&\hbox{Find $U_{\e,G} \in {\cal K}_{\e}$ such that   for every $v \in {\cal K}_{\e}$ },\\
&{\mathbf a}^{\e}(U_{\e,G},v-U_{\e,G})+ \big\langle G,|[v_{\tau_{\e}}]_{S_{\e}}|-|[(U_{\e,G})_{\tau_{\e}}]_{S_{\e}}|\big\rangle_\e \; \geq \left( f,v - U_{\e,G}\right) 
\end{aligned} \right.
\end{equation} where the linear form $G$ is an element of the cone
$$
{\cal C}^*_\e\doteq\big\{G\in \widetilde{H}^{-1/2}(S_\e)\;\;|\;\; \forall \phi\in  \widetilde{H}^{1/2}(S_\e),\;\; \phi\geq  0\;\Longrightarrow \; \big\langle G,\phi\big\rangle_\e\geq 0\big\}.
$$ 

The existence and uniqueness  of the solution of problem \eqref{var_stat_G-0} is obtained as the solution of a minimization convex functional (see Section \ref{S7}). 
\subsection{Estimates}
One has the following estimates:
\begin{prop}\label{TH71}
	The solution $U_{\e,G}$ of the contact problem  \eqref{var_stat_G-0} satisfies the following a priori estimates: 
	\begin{equation}\label{EstC1}
	\begin{aligned}
	&\|U_{\e,G}\|_{H^1(\O^*_\e)}\leq C\|f\|_{L^2(\O)},\\
	&\|[U_{\e,G}]_{S_\e}\|_{H^{1/2}(S_\e)}\leq C_0 \sqrt\e\|f\|_{L^2(\O)}.
	\end{aligned}
	\end{equation} 
	Furthermore, one has
	\begin{equation}\label{EstC10}
	\sigma_{\e}(U_{\e,G})_{|\partial {\cal S}^-_\e}\,{\nu_{\e}}=\sigma_{\e}(U_{\e,G})_{|\partial {\cal S}^+_\e}\,{\nu_{\e}}\qquad \hbox{in } \quad \widetilde{H}^{-1/2}(S_\e)^3
	\end{equation} and 
	\begin{equation}\label{EstC101}
	\sqrt\e \|\sigma_{\nu_\e}(U_{\e,G})\|_{H^{-1/2}(S_\e)}\le C\|f\|_{L^2(\O)}
	\end{equation}  where $\sigma_{\nu_\e}(U_{\e,G})\doteq(\sigma_{\e}(U_{\e,G})_{|\partial {\cal S}^+_\e}\,{\nu_{\e}}) \cdot {\nu_{\e}}$ and $-\sigma_{\nu_\e}(U_{\e,G})$ belongs to ${\cal C}^*_\e$. 
	\smallskip
	
\noindent 	The constants do not depend on $\e$.
\end{prop}
\begin{proof} 
	\noindent From inequality \eqref{var_stat_G-0}, one obtains
$${\mathbf a}^{\e}(U_{\e,G},U_{\e,G})+ \big\langle G,|[(U_{\e,G})_{\tau_{\e}}]_{S_{\e}}|\big\rangle_\e \leq \left( f,U_{\e,G}\right).$$ Then  estimates \eqref{EstC1} follow the ones of Lemma \ref{prunk}.
\smallskip
	
	Now, let $v$ be in $H^1_{\Gamma}(\O)^3$.  Since $[v_{\nu_{\e}}]_{S_\e}=0$, the fields $U_{\e,G}+ v$ and $U_{\e,G}-v$ belong to ${\cal K}_\e$, one has
	$${\mathbf a}^{\e}(U_{\e,G},U_{\e,G}\pm v-U_{\e,G})+ \big\langle G,|[(U_{\e,G}\pm v)_{\tau_{\e}}]|_{S_{\e}}-|[(U_{\e,G})_{\tau_{\e}}]|_{S_{\e}}\big\rangle_\e \; \geq \left( f,\pm v\right).$$
Since $v$  belongs to $H^1_{\Gamma}(\O)^3$
	$$[(U_{\e,G}\pm v)_{\tau_{\e}}]_{S_{\e}}=[(U_{\e,G})_{\tau_{\e}}]_{S_{\e}}.$$
	Hence
	$$
	{\mathbf a}^{\e}(U_{\e,G},v) \geq \left( f,v\right)\;\; \hbox{and}\;\; {\mathbf a}^{\e}(U_{\e,G},-v) \geq \left( f,-v\right).
	$$
	As a consequence
	\begin{equation}\label{EQ_55}
	\forall v\in H^1_{\Gamma}(\O)^3,\qquad {\mathbf a}^{\e}(U_{\e,G},v)= \left( f,v\right)
	\end{equation}
	which gives \eqref{EstC10}. Due to the estimates \eqref{EstSigma}-\eqref{EstC1}$_1$ and the fact that $\hbox{div}\big(\sigma_{\e}(U_{\e,G})\big)=f$ in ${\cal S}_\e$ we obtain \eqref{EstC101}.\\
	 To prove that $-\sigma_{\nu_{\e}}(U_{\e,G})$ belongs to ${\cal C}^*_\e$ consider a non-negative function $\phi \in H^{1/2}(\partial {\cal S})$ vanishing in $\partial {\cal S}\setminus S$.  Define $v$ as a lifting (vanishing on the boundary of $\Y$) of $\phi\,\nu$ in $H^1(\Y\setminus\overline{\cal S})$\footnote{ such a $v$ exists since  the boundary  of ${\cal S}$ is  ${\cal C}^{1,1}$} and by $0$ in ${\cal S}$. By construction $v \in H^1(\Y^*)^3$ and 
	 $$[v_{\tau}]_{S} =0,\qquad (v_{\nu})_{|\partial {\cal S}^+}=\phi,\quad (v_{\nu})_{|\partial {\cal S}^-}=0,\qquad [v_{\nu}]_{S}=\phi.$$ 
	 For $\xi$ every $\Xi_\e$, set
	$$
	v^\xi_\e=\left\{
	\begin{aligned}
	&\e\,v\Big(\Big\{{\cdot\over \e}\Big\}\Big)\quad \hbox{a.e. in }\;\; \e\xi+\e \Y,\\
	&0\hskip 19mm  \hbox{a.e. in }\;\; \O\setminus(\e\xi+\e \overline\Y).
	\end{aligned}\right.
	$$
	Taking $U_{\e,G}-v^\xi_\e$ as test displacement in \eqref{var_stat_G-0} and using the fact that $\hbox{div}\big(\sigma_{\e}(U_{\e,G})\big)=f$ in $\O^*_\e$, one obtains that $-\sigma_{\nu_{\e}}(U_{\e,G}) \in {\cal C}^*_\e$.
\end{proof}

\section{Regularized Coulomb friction problem}\label{S7}

In order to prove the existence of solutions to the problem \eqref{var_stat}, under a suitable assumption on the boundary of ${\cal S}$,  one can show (as in \cite{EJ05}) that there exists $\alpha\in (0,1/2)$ such that \begin{equation*}
\forall  G\in {\cal C}^*_\e\cap H^{-1/2+\alpha}(S_\e),\quad \mu |\sigma_{\nu_\e}(U_{\e,G})|\in {\cal C}^*_\e\cap H^{-1/2+\alpha}(S_\e)\;\;\hbox{and}\;\; \sqrt\e \|\sigma_{\nu_\e}(U_{\e,G})\|_{H^{-1/2+\alpha}(S_\e)}\le C\|f\|_{L^2(\O)}.
\end{equation*}
Then, the Schauder's theorem (see Theorem \ref{THSCH}) gives the existence of fixed points for the map 
$$A\,:\, G\in {\cal C}^*_\e\cap H^{-1/2+\alpha}(S_\e)\longmapsto \mu |\sigma_{\nu_\e}(U_{\e,G})|\in {\cal C}^*_\e\cap H^{-1/2+\alpha}(S_\e).$$ Thus, the problem \eqref{var_stat} admits solutions. Estimate \eqref{EQ35} yields
$$
\|{\cal T}^b_\e(\sigma_{\nu_\e}(U_{\e,G}))\|_{L^2(\O; \widetilde{H}^{-1/2+\alpha}(S))}=\e^{1/2}\|\sigma_{\nu_\e}(U_{\e,G})\|_{\widetilde{H}^{-1/2+\alpha}(S_\e)}\leq C\|f\|_{L^2(\O)}.
$$
For the homogenization process we need  the compactness of the sequence $\big\{{\cal T}^b_\e\big(\sigma_{\nu_\e}(U_{\e,G})\big)\big\}_\e$ in $L^2(\O; \widetilde{H}^{-1/2}(S))$. Unfortunately,  the above estimate is not sufficient. One must improve  \eqref{EstC101}, this could be obtain by comparing the norms of the tangential jumps $|[(U_{\e,G})_{\tau_\e}]_{S_{\e}}|$ in two neighboring cells.
But, it is well known that the following inequality:
$$\| |u|- |v|\|_{H^{1/2}(S_\e)}\leq C\|u-v \|_{H^{1/2}(S_\e)},\qquad \forall (u,v)\in H^1(S_\e)^3\times H^1(S_\e)^3$$ is false. Moreover, one can not replace the euclidian norm $|\cdot|$ by any kind of its approximation because, if one has
$$\|f(u)-f(v)\|_{H^{1/2}(S_\e)}\leq C\|u-v \|_{H^{1/2}(S_\e)},\qquad \forall (u,v)\in H^1(S_\e)^3\times H^1(S_\e)^3$$ with e.g. $f\in {\cal C}^{1}(\R^3;\R^3)$  then $f$ is affine! \\[1mm]
On the basis of Proposition \ref{prop111} (see Annex \ref{S11}) and in order to perform the homogenization process, we choose to modify the problem \eqref{var_stat} by adding a regularization term.\\[1mm]

We equip $H^2_\Gamma(\O^*_\e)^3$ with the following semi-norm:
$$\GN_\e(u)\doteq\sqrt{\|e(u)\|^2_{L^2(\O^*_\e)}+\e^2\|\nabla e(u)\|^2_{L^2(\O^*_\e)}}.$$ 
Since the displacements belonging to $H^2_\Gamma(\O^*_\e)^3$ 
vanish on $\Gamma$,  this semi-norm is a norm and thus $H^2_\Gamma(\O^*_\e)^3$ is a Hilbert space.\\
In view of Proposition \ref{TH71}, denote
$$
\GH_\e(\O)\doteq\big\{ u\in H^2_\Gamma(\O^*_\e)^3\;|\;  \;\; (\sigma^{\e}(u)_{|\partial {\cal S}^-_\e}\,{\nu_\e}) \cdot {\nu_\e}=(\sigma^{\e}(u)_{|\partial {\cal S}^+_\e}\,{\nu_\e}) \cdot {\nu_\e}\leq 0\;\; \; \hbox{a.e. in  } \; S_\e\Big\}.
$$
The set $\GH_\e(\O)$ is a convex closed subset of $H^2_\Gamma(\O^*_\e)^3$. \\

Observe that since the $a_{ijkl}$'s are $W^{1,\infty}$ in a neighborhood of $S$  and since the boundary of ${\cal S}$ is ${\cal C}^{1,1}$,  the traces on $\partial{\cal S}^\pm_\e$  of the stress tensor of the elements in $\GH_\e(\O)$ belong to $H^{1/2}(\partial {\cal S}^\pm_\e)$. \\[1mm]
For every $u\in \GH_\e(\O)$, we denote $\sigma_{\nu_\e}(u)\in H^{1/2}(S_\e)$ the restriction to $S_\e$ of  $(\sigma^{\e}(u)_{|\partial {\cal S}^+_\e}\,{\nu_\e}) \cdot {\nu_\e}$.\\[1mm]
{\bf From now on, in the left hand-side of problem \eqref{var_stat} we add the regularization term}\\[1mm]
$$\kappa\e^2[\,\nabla e(u_\e)\,, \,\nabla e(v-u_\e)\,]_\e$$ where
$$\forall (u,v)\in H^2_\Gamma(\O^*_\e)^3\times H^2_\Gamma(\O^*_\e)^3,\qquad [\nabla e(u)\,, \,\nabla e(v)]_\e\doteq \int_{\O^*_\e}\nabla e(u)\, \nabla e(v)\, dx.$$
We are therefore led to consider the following variational inequality:
\begin{equation}\label{var_Inq}
\left\{\begin{aligned} 
&\hbox{Find $u_\e \in \GH_\e(\O)\cap{\cal K}_{\e}$ such that   for every $v \in \GH_\e(\O)\cap{\cal K}_{\e}$ },\\
&{\mathbf a}^{\e}(u_\e,v-u_\e)+ \kappa\e^2[\,\nabla e(u_\e)\,, \,\nabla e(v-u_\e)\,]_\e + \big(\mu |\sigma_{\nu_\e}(u_\e)|,|[v_{\tau_\e}]_{S_\e}|-|[(u_\e)_{\tau_\e}]_{S_\e}|\big)_\e \; \geq \left( f,v - u_\e\right) 
\end{aligned} \right.
\end{equation}  where
$$
(\phi , \psi)_\e=\int_{S_\e}\phi\,\psi\, d\sigma_\e,\qquad \forall(\phi,\psi)\in L^2(S_\e)^2.
$$ 
\begin{lem}\label{lem1} For every $u$  in $\GH_\e(\O)$, one has
\begin{equation}\label{EQ54}
\sqrt \e\|e(u)_{|\partial {\cal S}^\pm_\e\cap S_\e}\|_{H^{1/2}(S_\e)}\leq C_1 \GN_\e(u)
\end{equation} and
\begin{equation}\label{EQ54-1}
\sqrt \e\|\sigma_{\nu_\e}(u)\|_{H^{1/2}(S_\e)}\leq C_1 \|\Ga\|_{W^{1,\infty}(S)}\GN_\e(u),
\end{equation}
where
$$
\|\Ga\|_{W^{1,\infty}(S)}=\max_{i,j,k,l=1}^3\|a_{ijkl}\|_{W^{1,\infty}(S)}.
$$
The constants do not depend on $\e$ and $\kappa$, they only depend  on $S$.
\end{lem}
\begin{proof} Let $\phi$ be in $H^1(\GY^*)$, the trace theorem gives
$$
\|\phi\|_{L^2({\cal S}^\pm\cap S)}\leq C\|\phi\|_{H^1(\GY)},\qquad \|\phi\|'_{H^{1/2}({\cal S}^\pm\cap S )}\leq C\|\nabla\phi\|_{L^2(\GY)}.
$$ Hence
\begin{equation}\label{EQ75}
\forall \Phi\in H^2({\cal S})^3,\quad \|e(\Phi)\|_{H^{1/2}(\partial {\cal S}^-)}\leq C_1\|e(\Phi)\|_{H^1({\cal S})}.
\end{equation}
The constants depend only on $S$.
Estimate \eqref{EQ54} follows  after  $\e$-scaling applied to the $e_{ij}(u)$'s. Then \eqref{EQ54-1} is an immediate consequence of \eqref{EQ54}.
\end{proof} 
\begin{cor} For every $u$  in $\GH_\e(\O)$
\begin{equation}\label{EQ55}
\|{\cal T}^b_\e\big(\sigma_{\nu_\e}(u)\big)\|_{L^2(\O\times S_\e)}\leq C \GN_\e(u).
\end{equation}
The constant does not depend on $\e$.
\end{cor}

\section{Existence results for $\e$  and friction fixed}\label{S8}

As a consequence of Lemma \ref{prunk}, one has
$$|(f,u)|\leq C\|f\|_{L^2(\O)} \| e(u) \|_{L^2 (\O^*_{\e})},\qquad \forall u\in \GH_\e(\O).$$

Now, consider the contact problem associated to \eqref{var_Inq} with constant friction 
\begin{equation}\label{var_stat_G}
\left\{\begin{aligned} 
&\hbox{Find $u \in \GH_\e(\O)\cap{\cal K}_{\e}$ such that   for every $v \in \GH_\e(\O)\cap{\cal K}_{\e}$ },\\
&{\mathbf a}^{\e}(u,v-u)+ \kappa\e^2[\,\nabla e(u)\,, \,\nabla e(v-u)\,]_\e+ \big(G,|[v_{\tau_\e}]_{S_\e}|-|[u_{\tau_\e}]_{S_\e}|\big)_\e \; \geq \left( f,v - u\right)
\end{aligned} \right.
\end{equation}  where the linear form $G$ is an element of the cone
$$
{\cal C}^{**}_\e\doteq\big\{G\in L^2(S_\e)\;\;|\;\; G\geq 0 \;\;\hbox{a.e. on }\; S_\e \big\}.
$$ 
The functional $J_\e$ defined in $\GH_\e(\O)\cap {\cal K}_\e$ by
$$J_\e\;:\;  u\in \GH_\e(\O)\cap {\cal K}_\e\longmapsto {1\over 2}{\mathbf a}^{\e}(u,u )+ {\kappa\e^2\over 2}[\,\nabla e(u)\,, \,\nabla e(u)\,]_\e+ \big(G,|[u_{\tau_\e}]_{S_\e}|\big)_\e -\left( f,u\right) $$ is strictly convex, weakly lower semicontinuous and due to \eqref{eq315}$_{1,2}$, it satisfies
$$\lim_{u\in \GH_\e(\O)\cap {\cal K}_\e,\, \GN_\e(u)\to +\infty} J_\e(u)=+\infty.$$
As a consequence, there exists a unique solution  for the  corresponding minimization problem or equivalently for the problem \eqref{var_stat_G} (see \cite{EJ05}). 
\section{Estimates}
In this section, the solution to problem \eqref{var_stat_G} is denoted $u_{\e,G}$.
\begin{thm}\label{TH71-bis}
	The solution $u_{\e,G}$ of the contact problem  \eqref{var_stat_G} satisfies 
	\begin{equation}\label{EstC1-bis}
	\begin{aligned}
	&\sqrt{\kappa}\e\|\nabla e(u_{\e,G})\|_{L^2(\O^*_\e)}\leq C\|f\|_{L^2(\O)},\\
	&\|u_{\e,G}\|_{H^1(\O^*_\e)} \leq C\|f\|_{L^2(\O)},\\
	&\|[u_{\e,G}]_{S_\e}\|_{H^{1/2}(S_\e)}\leq C\sqrt\e \|f\|_{L^2(\O)}.
	\end{aligned}
	\end{equation} 
	Furthermore, one has
	\begin{equation}\label{EstC101-bis}
	\sqrt\e \|\sigma_{\nu_\e}(u_{\e,G})\|_{H^{1/2}(S_\e)}\leq C_2\Big(1+{1\over \sqrt\kappa}\Big)\|f\|_{L^2(\O)}.
	\end{equation}
Moreover, the solution $u_{\e,G}$   depends continuously on the given friction $G$ and one has
	\begin{equation}\label{EstC11-bis}
\|\sigma_{\nu_\e}(u_{\e,G_1})-\sigma_{\nu_\e}(u_{\e,G_2})\|_{H^{1/2}(S_\e)}\leq {C_0C_1\|\Ga\|_{W^{1,\infty}(S)}\over \min\{\overline{\alpha},\kappa\}}\|G_1-G_2\|_{L^2(S_\e)},\qquad \forall (G_1,G_2)\in({\cal C}^{**}_\e)^2.
\end{equation}
	The constants do not depend on $\e$ and $\kappa$.
\end{thm}
\begin{proof} 
	\noindent From equality \eqref{var_stat_G}, one obtains
$${\mathbf a}^{\e}(u_{\e,G},u_{\e,G})+\kappa\e^2[\,\nabla e(u_{\e,G})\,, \,\nabla e(u_{\e,G})\,]_\e+ \big(G,|[(u_{\e,G})_{\tau_\e}]_{S_\e}|\big)_\e \leq \left( f,u_{\e,G}\right).$$ Then  estimates \eqref{EstC1-bis} follow the ones of Proposition \ref{prunk} while \eqref{EstC101-bis} is a consequence of \eqref{EQ54-1}.
\smallskip

\noindent Now,  problem \eqref{var_stat_G} with $G_1$ and then $G_2$ together with the estimates \eqref{eq315}$_2$  gives
	$$
	\begin{aligned}
	\overline{\alpha}&\|e(u_{\e,G_1})-e(u_{\e,G_2})\|^2_{L^2(\O^*_\e)} +\kappa\e^2[\,\nabla e(u_{\e,G_1}-u_{\e,G_2})\,, \,\nabla e(u_{\e,G_1}-u_{\e,G_2})\,]_\e\\
	& \leq \|G_1-G_2\|_{L^2( S_\e)}\big(\|[(u_{\e,G_1}-u_{\e,G_2})_{\tau_\e}]_{ S_\e}\|_{L^2( S_\e)}\big)\\
	&\leq C_0 \sqrt\e\|G_1-G_2\|_{L^2(S_\e)}\|e(u_{\e,G_1})-e(u_{\e,G_2})\|_{L^2(\O^*_\e)}.
	\end{aligned}
	$$
So
\begin{equation}\label{EQ64}
\GN(u_{\e,G_1}-u_{\e,G_2})\leq {C_0\over \min\{\overline{\alpha},\kappa\}}\sqrt\e\|G_1-G_2\|_{L^2(S_\e)}.
\end{equation}
Then  estimate \eqref{EQ54-1} yields
$$
\|\sigma_{\nu_\e}(u_{\e,G_1})-\sigma_{\nu_\e}(u_{\e,G_2})\|_{H^{1/2}(S_\e)}\leq {C_1 \|\Ga\|_{W^{1,\infty}(S)}\over \sqrt\e}\GN\big(u_{\e,G_1}-u_{\e,G_2}\big)\leq {C_0C_1\|\Ga\|_{W^{1,\infty}(S)}\over \min\{\overline{\alpha},\kappa\}}\|G_1-G_2\|_{L^2(S_\e)}.
$$ Therefore  \eqref{EstC11-bis} follows. 
\end{proof}

\subsection{Solution of the Coulomb's problem}

A solution of the  contact problem with Coulomb friction is characterized by a fix point of the operator
\begin{equation}\label{A}
A\;:\; G\in {\cal C}^*_\e\longmapsto \mu |\sigma_{\nu_\e}(u_{\e,G})|\in {\cal C}^*_\e.
\end{equation}
The existence of a fix point is proved by the Banach fixed-point theorem.

\begin{prop}\label{prop53}  Assume 
\begin{equation}\label{Contraction}
{C_0C_1\|\Ga\|_{W^{1,\infty}(S)}\|\mu\|_{L^\infty(\O)}\over \min\{\overline{\alpha},\kappa\}}<1
\end{equation}
then problem  \eqref{var_Inq} admits a unique solution. The constant $C_0C_1$ depends only on $S$.
\end{prop}
	
\begin{proof} From estimate \eqref{EstC11-bis} one gets
$$
\begin{aligned}
\|\mu \big(\sigma_{\nu_\e}(u_{\e,G_1})-\sigma_{\nu_\e}(u_{\e,G_2})\big)\|_{L^2(S_\e)}\leq &\|\mu\|_{L^\infty(\O)}\|\sigma_{\nu_\e}(u_{\e,G_1})-\sigma_{\nu_\e}(u_{\e,G_2})\|_{L^2(S_\e)}\\
\leq & {C_0C_1\|\Ga\|_{W^{1,\infty}(S)}\|\mu\|_{L^\infty(\O)}\over \min\{\overline{\alpha},\kappa\}}\|G_1-G_2\|_{L^2(S_\e)}.
\end{aligned}
$$ Now, consider the map $A$ defined by \eqref{A}. It is continuous and if the condition \eqref{Contraction} is satisfied then the Banach fixed-point theorem gives a unique solution to the Coulomb friction problem  \eqref{var_Inq}.
\end{proof}
Solutions of \eqref{var_Inq} can also be obtained by Schauder's theorem.
\begin{thm}[Schauder's theorem]\label{THSCH} Let ${\cal X}$ be a Banach space and let ${\cal Y} \subset {\cal X}$ be a convex compact subset. Then every continuous map (for the strong topology of ${\cal X}$) $A:\ {\cal Y} \to {\cal Y}$  has a fixed point in  ${\cal Y}$.
\end{thm}
\begin{prop}\label{prop53}  For every $\e$ the problem  \eqref{var_Inq} admits solutions. The solutions satisfy
$$\sqrt\e \|\mu \,\sigma_{\nu_\e}(u_{\e,G})\|_{H^{1/2}(S_\e)}\leq C_1C_2\Big(1+{1\over \sqrt\kappa}\Big)\|\mu\|_{W^{1,\infty}(\O)}\|f\|_{L^2(\O)}.
$$
The constant does not depend on $\e$.
\end{prop}
	
\begin{proof} Since $\mu$ belongs to $W^{1,\infty}(\O)$ and  due to \eqref{EstC1-bis}-\eqref{EstC101-bis}, one has
$$
\sqrt\e \|\mu \,\sigma_{\nu_\e}(u_{\e,G})\|_{H^{1/2}(S_\e)}\leq C_1\sqrt \e\|\mu\|_{W^{1,\infty}(\O)}\|\sigma_{\nu_\e}(u_{\e,G})\|_{H^{1/2}(S_\e)}\leq C_1C_2\Big(1+{1\over \sqrt\kappa}\Big)\|\mu\|_{W^{1,\infty}(\O)}\|f\|_{L^2(\O)}.
$$
Now, set
\begin{equation}\label{BoundR}
R=C_1C_2\Big(1+{1\over \sqrt\kappa}\Big)\|\mu\|_{W^{1,\infty}(\O)}\|f\|_{L^2(\O)}.
\end{equation}	
Therefore, one has
$$
\forall G\in  {\cal C}^{**}_\e,\quad  \mu \,\sigma_{\nu_\e}(u_{\e,G})\hbox{ belongs to } {\cal C}^{**}_\e\;\; \hbox{and}\quad  \|G\|_{H^{1/2}(S_\e)}\leq {R\over \sqrt \e}\,\Longrightarrow\, \|\mu \,\sigma_{\nu_\e}(u_{\e,G})\|_{H^{1/2}(S_\e)}\leq {R\over \sqrt \e}. 
$$
Applying Schauder's  theorem with the map $A$ (see \eqref{A}), ${\cal X}=L^2(S_\e)$ (endowed with the strong topology) and choosing  $\ds {\cal Y}=\Big\{G \in {\cal C}^{**}_\e\;\; |\;\;  \|G\|_{H^{1/2}(S_\e)}\le {R\over \sqrt \e}\Big\}$ (which is a compact convex subset of ${\cal X}$)  give solutions to problem  \eqref{var_Inq}.
\end{proof}

\section {Homogenization process}
{\it In this section, we denote $u_\e$ the solution to problem \eqref{var_Inq}. }\\[1mm]
	From \eqref{EstC1-bis} the displacement $u_\e$ satisfies the following estimates:
\begin{equation}\label{EstC1-ter}
	\begin{aligned}
	&\e\|\nabla e(u_\e)\|_{L^2(\O^*_\e)}\leq C\|f\|_{L^2(\O)},\\
	&\|u_\e\|_{H^1(\O^*_\e)} \leq C\|f\|_{L^2(\O)},\\
	&\|[u_\e]_{S_\e}\|_{H^{1/2}(S_\e)}\leq C\sqrt\e \|f\|_{L^2(\O)},\\
	&\sqrt \e\|\sigma_{\nu_\e}(u_\e)\|_{H^{1/2}(S_\e)}\leq C\|f\|_{L^2(\O)}.
	\end{aligned}
	\end{equation} The constants do not depend on $\e$.
\subsection{A compactness result}

Below we give a result related to the unfolding method.
\begin{definition} Let $\o$ be an open set strictly included in  $\O$. For every $\e\leq \hbox{dist}(\o,\partial \O)/4$ and for every $\phi\in L^p(\O^*_\e)$, $p\in [1,+\infty]$, we define  ${\cal Q}^{\diamond}_\e (\phi)\in W^{1,\infty}(\o ; L^p(\Y^*))$ by $Q_1$ interpolation. We set 
$$
\begin{aligned}
&\hbox{for } \xi\in \Xi_\e,\quad\Qed (\phi)(\e\xi, y)=\Tes (\phi)(\e\xi,y)\qquad  \hbox{for a.e. }y\in \Y^*,\\
&\hbox{for } x\in \o,\quad\Qed (\phi)(x, y)\;\hbox{is the $Q_1$ interpolate of $\Qed (\phi)$ at the vertices of the cell $\ds\Big(\e\Big[{x\over \e}\Big]_\e+\e \Y\Big)\times \{y\}$,}\\
&\hskip 28mm \hbox{for a.e. }y\in \Y^*.
\end{aligned}
$$
\end{definition}
\begin{lem} Assume $\o\Subset \O$. If $\e$ is small enough,  for every  $\phi\in H^1(\O^*_\e)$ then   $\Qed (\phi)$  belongs to  $H^1(\o; H^1( \Y^*))$ and
$$\| \Qed (\phi)\|_{H^1(\o; H^1( \Y^*))}\leq C \| \phi\|_{H^1(\O^*_\e)}.$$ The constant does not depend on $\e$ and $\o$.
\end{lem}
\begin{proof} See  \cite[Propositions 2.6-2.7]{cddgz} for the proof.
\end{proof}
\begin{lem} Suppose  $\o\Subset \O$. One has
\begin{equation}\label{EQ68}
\|\Qed \big(e(u_\e)\big)\|_{H^1(\o; H^1(\Y^*))}\leq C\|f\|_{L^2(\O)},\qquad \|\Qed \big(e(u_\e)\big)-\Tes\big(e(u_\e)\big)\|_{L^2(\o  ; H^1( \Y^*))}\leq C\e\|f\|_{L^2(\O)}.
\end{equation}
The constants do not depend on $\e$ (it depends on $\o$).
\end{lem}
\begin{proof} For every function  $\phi$ belonging to $L^2(\O^*_\e)$ with support strictly included in $\O$ and for $\e$ small enough, set
$$\Delta_\gk \phi=\phi(\cdot+\e\gk)-\phi,\qquad  \Gk=\Ge_i,\;\; i\in \{1,2,3\},$$
$\Delta_\gk \phi$ also belongs to $L^2(\O^*_\e)$.\\[1mm]
\noindent {\it Step 1.} An equicontinuity result.\\
Let $\o$ be an open set such that  $\o\Subset \O'\Subset \O$ (see the assumption on $\mu$) with $\hbox{dist}(\o,\partial\O)=\d>0$. There exists a function $\rho_{\o}\in {\cal D}(\O)$ such that 
$$
\begin{aligned}
&0\leq \rho_{\o}(x)\leq 1,\quad \forall x\in \O,\\
& \rho_{\o}(x) =0 \quad \hbox{if}\; \hbox{dist}(x,\partial \O)\leq {\d\over 4},\quad x\in \O,\\
& \rho_{\o}(x) =1 \quad \hbox{if}\; \hbox{dist}(x,\partial \O)\geq {\d\over 2},\quad x\in \O.
\end{aligned}
$$ Function $\rho_{\o}$ satisfies ($(i,j)\in \{1,2,3\}^2$)
\begin{equation}\label{EQ90}
\Big\|{\partial \rho_{\o}\over \partial x_i} \Big\|_{L^\infty(\O)}\leq {C\over \d},\quad \Big\|{\partial^2 \rho_{\o}\over \partial x_i\partial x_j} \Big\|_{L^\infty(\O)}\leq {C\over \d^2},\quad \Big\|\Delta_{\gk}{\partial \rho_{\o}\over \partial x_i} \Big\|_{L^\infty(\O)}\leq C{\e\over \d^2},\quad \Big\|\Delta_{\gk}{\partial^2 \rho_{\o}\over \partial x_i\partial x_j} \Big\|_{L^\infty(\O)}\leq C{\e\over \d^3}
\end{equation} where the constant depends on $\partial \O$. Below, in the estimates we will not mention the dependence of constants with respect to $\d$ since the open set $\o$ is fixed. \\[1mm]
Let $v$ be in  ${\cal K}_\e$. Set $v_\e= \ u_{\e}+\rho_{\o} (v-\rho_{\o} u_{\e})$. First, observe that $v=u_{\e}$ in the neighborhood of the boundary of $\O$, also note that $v_\e$ is an admissible test displacement (it belongs to ${\cal K}_\e$). Denote		
	$$\gu_\e=\rho_{\o} u_{\e},\quad \gf=\rho_{\o} f.$$ 
In \eqref{var_Inq} choose as test function $v_\e$,  rewrite the inequality in terms of $\gu_\e$ and $\gf$,  shifting the terms with derivatives of $\rho_{\o}$ into the right-hand side and additionally  due to the assumption on $\mu $ observe that
$$ \mu  \sigma_{\nu_\e}(\gu_\e)=\mu  \sigma_{\nu_\e}(u_\e).$$ One obtains
\begin{equation}\label{EQ93}
\begin{aligned}
& {\mathbf a}^\e\big(\gu_\e, v-\gu_\e\big)+\kappa\e^2\big[\gu_\e,v-\gu_\e\big]_\e- \big(\mu\,  \sigma_{\nu_\e}(\gu_\e) , |[v_{\tau_\e}]|_{S_\e}-|[(\gu_\e)_{\tau_\e}]|_{S_\e}\big)_\e\\
&\geq \int_{\O^*_\e}\gf\big(v- \gu_\e\big)dx+\gb_\e(u_\e,v-\gu_\e) \qquad \hbox{for every $v\in {\cal K}_\e$},
\end{aligned}
\end{equation} 
where $\gb_\e$ is a  bilinear form. For every $w \in H^1_\Gamma(\O^*_\e)^3$ one has 
\begin{equation*} 
\gb_\e(u_\e, w)= {\mathbf a}^\e\big(\rho_{\o} u_\e,w\big)-{\mathbf a}^\e\big(u_\e,\rho_{\o} w\big)+\kappa\e^2\big(\big[\rho_{\o} u_\e,w\big]_\e-\big[u_\e,\rho_{\o} w\big]_\e\big)
\end{equation*} 
which is also equal to
\begin{equation} \label{EQ94}
\begin{aligned}
&	\gb_\e(u_\e,w)= \int_{\O^*_\e} a^\e_{ijkl}\Big[e_{ij}\big(u_{\e}\big){\partial\rho_{\o}\over  \partial x_l}\,w_k-{\partial\rho_{\o}\over  \partial x_j} \big(u_{\e}\big)_i\,e_{kl}(w)\Big]\,dx\\
&	+ \kappa \e^2 \int_{\O^*_\e}\Big[\Big({\partial^2 \rho_{\o}\over  \partial x_l\partial x_i}\big(u_{\e}\big)_j+ 2{\partial\rho_{\o}\over  \partial x_l}e_{ij}(u_\e)\Big)\,{\partial e_{ij}(w)\over  \partial x_l} -{\partial e_{ij}(u_{\e})\over  \partial x_l}\Big({\partial^2 \rho_{\o}\over  \partial x_l\partial x_i}w_j+ 2{\partial\rho_{\o}\over  \partial x_l}\, \,e_{ij}(w)\Big)\Big]\,dx,\\
\end{aligned}
\end{equation} 
For $\e$ small enough, the set $(\e\Gk+\e S_\e)\cap \o$ is included in $\O$ and one has $\gu_{\e}(\cdot\pm\e\Gk) \in \GH_\e(\O)\cap {\cal K}_\e$.\\
 Choose $v=\gu_{\e}(\cdot-\e\Gk) \in \GH_\e(\O)\cap {\cal K}_\e$ as test displacement in \eqref{EQ93}. Due to the periodicity of the coefficients $a_{ijkl}$, $\mu$ and  after a change of variables, one obtains
\begin{equation}\label{EQ91}
\begin{aligned}
&{\mathbf a}^{\e}(\gu_{\e}(\cdot+\e\Gk),\gu_\e-\gu_{\e}(\cdot+\e\Gk))+ \kappa\e^2[\,\nabla e(\gu_{\e}(\cdot+\e\Gk))\,, \,\nabla e(\gu_\e-\gu_{\e}(\cdot+\e\Gk))\,]_\e\\
+& \big(\mu(\cdot+\e\Gk) |\sigma_{\nu_\e}\big(\gu_\e(\cdot+\e\Gk)\big)|\,,\,|[(\gu_\e)_{\tau_\e}]_{S_\e}|-|[(\gu_{\e}(\cdot+\e\Gk))_{\tau_\e}]_{S_\e}|\big)_\e \\
& \geq \left( \gf(\cdot+\e\Gk),\gu_\e - \gu_{\e}(\cdot+\e\Gk)\right)+\gb_\e(u_{\e}(\cdot+\e\Gk) , \gu_\e-\gu_{\e}(\cdot+\e\Gk)). 
\end{aligned}
\end{equation} 
Besides, since $\gu_{\e}(\cdot+\e\Gk) \in \GH_\e(\O)\cap {\cal K}_\e$, from \eqref{EQ93} one also has
\begin{equation}\label{EQ92}
\begin{aligned}
&{\mathbf a}^{\e}(\gu_{\e},\gu_{\e}(\cdot+\e\Gk)-\gu_{\e})+ \kappa\e^2[\,\nabla e(\gu_{\e})\,, \,\nabla e(\gu_{\e}(\cdot+\e\Gk)-\gu_\e)\,]_\e\\
+ &\big(\mu\, \sigma_{\nu_\e}(\gu_\e)\,,\, |[(\gu_{\e}(\cdot+\e\Gk))_{\tau_\e}]_{S_\e}|-|[(\gu_{\e})_{\tau_\e}]_{S_\e}|\big)_\e  \\
& \geq \left( \gf,\gu_{\e}(\cdot+\e\Gk) - \gu_{\e}\right) +\gb_\e(u_{\e} , \gu_{\e}(\cdot+\e\Gk)-\gu_\e)
\end{aligned}
\end{equation} 
The above inequalities \eqref{EQ91}-\eqref{EQ92} lead to
\begin{equation}\label{EQ92-2}
\begin{aligned}
&{\mathbf a}^{\e}(\Delta_\gk \gu_{\e},\Delta_\gk \gu_{\e})+\kappa\e^2\|\nabla e(\Delta_\gk \gu_{\e})\|^2_{L^2(\O^*_\e)}+\big(\mu\, \sigma_{\nu_\e}(\Delta_\gk \gu_\e)\,, \,\Delta_\gk|[(\gu_{\e})_{\tau_\e}]_{S_\e}|\big)_\e \\
\leq &\left( \Delta_\gk \gf,\Delta_\gk \gu_{\e}\right)+|\Delta_{\gk}\gb_\e|+|\big(\Delta_\gk\mu\, |\sigma_{\nu_\e}(\gu_\e)(\cdot+\e\Gk)|\,,\,\Delta_\gk|[(\gu_\e)_{\tau_\e}]_{S_\e}|\big)_\e|
\end{aligned}
\end{equation} 
where $\Delta_{\gk}\gb_\e$ is equal to (thanks to \eqref{EQ94})
\begin{equation}\label{EQ910-}
	\begin{aligned}
	\Delta_{\gk} {\bf b}_\e\doteq&\int_{\O^*_\e} a^\e_{ijkl}\Big[\Delta_\gk\Big( e_{ij}\big(u_{\e}\big){\partial\rho_{\o}\over  \partial x_l}\Big)\,\big(\Delta_\gk \gu_\e)_k\, dx-\int_{\O^*_\e}\Delta_\gk\Big({\partial\rho_{\o}\over  \partial x_j} \big(u_{\e}\big)_i\Big)\,e_{kl}\big(\Delta_\gk \gu_\e\big)\Big]\,dx\\
	&+ \kappa \e^2 \int_{\O^*_\e}\Big[\Delta_\gk\Big({\partial^2 \rho_{\o}\over  \partial x_l\partial x_i}\big(u_{\e}\big)_j+ 2{\partial\rho_{\o}\over  \partial x_l}e_{ij}(u_\e)\Big)\,{\partial e_{ij}(\Delta_\gk \gu_\e)\over  \partial x_l}\, dx\\
	&- \kappa \e^2 \int_{\O^*_\e}\Big[\Delta_\gk\Big({\partial e_{ij}(u_\e)\over  \partial x_l}{\partial^2 \rho_{\o}\over  \partial x_l\partial x_i}\Big)\Delta_\gk\big(\gu_{\e}\big)_j+ 2\Delta_\gk\Big({\partial e_{ij}(u_\e)\over  \partial x_l}{\partial\rho_{\o}\over  \partial x_l}\Big)e_{ij}(\Delta_\gk\gu_\e)\Big]\, dx.
	\end{aligned}
\end{equation}
Estimates \eqref{EstC1-ter}$_4$ and \eqref{JumpS} give
\begin{equation}\label{EQ910}
\begin{aligned}
|\big(\Delta_\gk\mu\, |\sigma_{\nu_\e}(\gu_\e)(\cdot+\e\Gk)|\,,\,\Delta_\gk|[(\gu_\e)_{\tau_\e}]_{S_\e}|\big)_\e| & \leq C\e \|\nabla\mu\|_{L^\infty(\O)}{1\over \sqrt\e}\|f\|_{L^2(\O)}\|\Delta_\gk|[(\gu_\e)_{\tau_\e}]_{S_\e}|\|_{L^2(S_\e)}\\
& \leq C\e \|\nabla\mu\|_{L^\infty(\O)}\|f\|_{L^2(\O)}\|e(\Delta_\gk \gu_\e)|\|_{L^2(\O^*_\e)}.
\end{aligned}
\end{equation}
Now, observe that the first and the last integrals in \eqref{EQ910-} are equal to
$$
\begin{aligned}
&\int_{\O^*_\e} \hskip-1mm a^\e_{ijkl}\,\Delta_\gk\Big( e_{ij}\big(u_{\e}\big){\partial\rho_{\o}\over  \partial x_l}\Big)\,\big(\Delta_\gk \gu_\e )_k\, dx
=\int_{\O^*_\e} \hskip-1mm a^\e_{ijkl}e_{ij}\big(u_{\e}\big){\partial\rho_{\o}\over  \partial x_l}\Delta_{-\gk}( \Delta_{\gk} \gu_\e)_k\,dx,\\
&\int_{\O^*_\e}\Big[\Delta_\gk\Big({\partial e_{ij}(u_\e)\over  \partial x_l}{\partial^2 \rho_{\o}\over  \partial x_l\partial x_i}\Big)\Delta_\gk\big(\gu_{\e}\big)_j+ 2\Delta_\gk\Big({\partial e_{ij}(u_\e)\over  \partial x_l}{\partial\rho_{\o}\over  \partial x_l}\Big)e_{ij}(\Delta_\gk\gu_\e)\Big]\, dx\\
=&\int_{\O^*_\e}\Big[\Big({\partial e_{ij}(u_\e)\over  \partial x_l}{\partial^2 \rho_{\o}\over  \partial x_l\partial x_i}\Big)\Delta_{-\gk}\big(\Delta_\gk\big(\gu_{\e}\big)_j\big)+ 2\Big({\partial e_{ij}(u_\e)\over  \partial x_l}{\partial\rho_{\o}\over  \partial x_l}\Big)\Delta_{-\gk}\big(e_{ij}(\Delta_\gk\gu_\e)\big)\Big]\, dx.
\end{aligned}
$$
Hence, using estimates \eqref{EQ90}$_1$ and \eqref{EstC1-ter}
$$
\begin{aligned}
&\Big|\int_{\O^*_\e} \hskip-1mm a^\e_{ijkl}\,\Delta_\gk\Big( e_{ij}\big(u_{\e}\big){\partial\rho_{\o}\over  \partial x_l}\Big)\,\big(\Delta_\gk \gu_\e )_k\, dx\Big|
\leq C\|f\|_{L^2(\O)}\|\Ga\|_{W^{1,\infty}(\O)}\|\Delta_{-\gk}( \Delta_{\gk} \gu_\e)\|_{L^2(\O^*_\e)}\\
&\Big|\int_{\O^*_\e}\Big[\Delta_\gk\Big({\partial e_{ij}(u_\e)\over  \partial x_l}{\partial^2 \rho_{\o}\over  \partial x_l\partial x_i}\Big)\Delta_\gk\big(\gu_{\e}\big)_j+ 2\Delta_\gk\Big({\partial e_{ij}(u_\e)\over  \partial x_l}{\partial\rho_{\o}\over  \partial x_l}\Big)e_{ij}(\Delta_\gk\gu_\e)\Big]\, dx\Big|\\
\leq&C\|f\|_{L^2(\O)}\big(\|\Delta_{-\gk}\big(\Delta_\gk(\gu_{\e})\big)\big)\|_{L^2(\O^*_\e)}+ \|\Delta_{-\gk}\big(e_{ij}(\Delta_\gk\gu_\e)\big)\|_{L^2(\O^*_\e)}\big).
\end{aligned}
$$
Therefore, with \eqref{EQ90}$_{2,3}$ and again estimates \eqref{EstC1-ter} one obtains
$$ 
\Big\|\Delta_\gk\Big({\partial\rho_{\o}\over  \partial x_j} \big(u_{\e}\big)_i\Big)\Big\|_{L^2(\O^*_\e\cap \o)}\leq C\e\|f\|_{L^2(\O)},\quad \Big\|\Delta_\gk\Big({\partial^2 \rho_{\o}\over  \partial x_l\partial x_i}\big(u_{\e}\big)_j+ 2{\partial\rho_{\o}\over  \partial x_l}e_{ij}(u_\e)\Big)\Big\|_{L^2(\O^*_\e\cap \o)}\leq C\|f\|_{L^2(\O)}.
$$ The constants only depend on $\o$. 
Besides, 
$$
\begin{aligned}
&\|\Delta_{-\gk}\big(\Delta_\gk(\gu_{\e})\big)\|_{L^2(\O^*_\e)}\leq C\e \|\nabla\big(\Delta_\gk(\gu_{\e})\big)\|_{L^2(\O^*_\e)}\leq C\e \|e\big(\Delta_\gk(\gu_{\e})\big)\|_{L^2(\O^*_\e)},\\
&\|\Delta_{-\gk}\big(e(\Delta_\gk\gu_\e)\big)\|_{L^2(\O^*_\e)}\|\leq C\e\|\nabla\big(\Delta_\gk e(\gu_{\e})\big)\|_{L^2(\O^*_\e)}\leq C\e \|\nabla \big(e(\Delta_\gk \gu_{\e})\big)\|_{L^2(\O^*_\e)}.
\end{aligned}
$$	
Summarizing the above equalities and estimates lead to
\begin{equation}\label{EQ99}
\begin{aligned}
|\Delta_{\gk} {\bf b}_\e|  &\leq C\e \|f\|_{L^2(\O)}\big( \|e(\Delta_\gk \gu_{\e})\|_{L^2(\O_\e^*)}+\e\|\nabla \big(e(\Delta_\gk \gu_{\e})\big)\|_{L^2(\O^*_\e)}\big).
\end{aligned}
\end{equation}
Hence,  \eqref{EQ92-2} together with  \eqref{EQ910}-\eqref{EQ99}  yield
$$
\begin{aligned}
&\overline{\alpha}\|e\big(\Delta_\gk u_{\e}\big)\|^2_{L^2(\O^*_\e)}+\kappa\e^2\|\nabla e(\Delta_\gk u_{\e})\|^2_{L^2(\O^*_\e)}\\
\leq & \e\|\mu\|_{L^\infty(\O)}\|\sigma_{\nu_\e}(\Delta_\gk u_\e)\|_{L^2(S_\e)}\|\Delta_\gk [(u_\e)_{\tau_\e}]\|_{L^2(S_\e)}\\
+&C\|f\|_{L^2(\O)}\|\Delta_{-\gk}\big(\Delta_\gk u_{\e}\big)\|_{L^2(\O^*_\e)}+C \e\|f\|_{L^2(\O)}\GN_\e(\Delta_\gk u_{\e}).
\end{aligned}
$$
Using \eqref{JumpS}, \eqref{EQ54-1}  that gives 
$$
\begin{aligned}
&\overline{\alpha}\|e\big(\Delta_\gk u_{\e}\big)\|^2_{L^2(\O^*_\e)}+\kappa\e^2\|\nabla e(\Delta_\gk u_{\e})\|^2_{L^2(\O^*_\e)}\\
\leq & \|\mu\|_{L^\infty(\O)}{C_1\over \sqrt\e} \|\Ga\|_{W^{1,\infty}(S)}\GN_\e(\Delta_\gk u_{\e})\,C_0\sqrt\e \|e\big(\Delta_\gk u_{\e}\big)\|_{L^2(\O^*_\e)}+C \e\|f\|_{L^2(\O)}\GN_\e(\Delta_\gk u_{\e}).
\end{aligned}
$$
Thus
$$
\begin{aligned}
&\min\{\overline{\alpha},\kappa\}\GN^2_\e(\Delta_\gk u_{\e}) \leq  C_0C_1\|\mu\|_{L^\infty(\O)}\|\Ga\|_{W^{1,\infty}(S)}\GN^2_\e(\Delta_\gk u_{\e})+C\e\|f\|_{L^2(\O)}\GN_\e(\Delta_\gk u_{\e}).
\end{aligned}
$$
At this point one deduces that under assumption \eqref{Contraction}, estimate 
\begin{equation}
\label{EQ67}
\GN_\e(\Delta_\gk u_{\e})\leq C\e\|f\|_{L^2(\O)},\qquad \Gk=\Ge_i,\;\; i\in \{1,2,3\}
\end{equation} holds true.
The constant does not depend on $\e$ (it depends on $\o$).\\

\noindent{\it Step 2.} We prove the estimates of the lemma.

\noindent First the above estimate \eqref{EQ67} and \eqref{EstC1-ter}$_1$ lead to
$$
\begin{aligned}
&\|\Tes\big(e(u_\e)\big)\|_{L^2(\O; H^1(\Y^*))}\leq C\|f\|_{L^2(\O)},\\
&\sum_{i=1}^3\| \Tes\big(e(u_\e)(\cdot+\e \Ge_i, \cdot)\big)-\Tes\big(e(u_\e)\big)\|_{L^2(\o; H^1(\Y^*))}\leq C\e\|f\|_{L^2(\O)}
\end{aligned}
$$ which in turn  yield  \eqref{EQ68} (see \cite{cdg1}).
\end{proof}
As a consequence of the above lemma one has
\begin{equation}\label{EQ912}
\|\Qed \big(e(u_\e)\big)-\Tes\big(e(u_\e)\big)\|_{L^2(\o  ; H^{1/2}(\partial {\cal S}^\pm))}\leq C\e\|f\|_{L^2(\O)}.
\end{equation}
The constants do not depend on $\e$ (it depends on $\o$).
\begin{lem}\label{lem65} There exists a constant $C^*$ which only depends on $\Y^*$ such that for every $(v,\widehat{v})\in H^1(\O)^3\times L^2(\O; H^1_{per}(\Y^*)^3\cap W(\Y^*))$ 
$$C^*\big(\|e(v)\|_{L^2(\O)}+\|e_y(\widehat{v})\|_{L^2(\O\times \Y^*)}\big)\leq \|e(v)+ e_y(\widehat{v})\|_{L^2(\O\times \Y^*)}.$$
\end{lem}
\begin{proof} Let $\zeta$ be a $3\times 3$ symmetric matrix and $\widehat{w}\in H^1_{per}(\Y^*)^3\cap W(\Y^*)$, one first proves
\begin{equation}\label{EQ69}
c^*\big(|\zeta|+\|e_y(\widehat{w})\|_{L^2(\Y^*)}\big)\leq \|\zeta+ e_y(\widehat{w})\|_{L^2(\Y^*)}\leq |\zeta|+\|e_y(\widehat{w})\|_{L^2(\Y^*)}.
\end{equation}
The right hand-side inequality is obvious.\\
To prove the left hand-side, apply the Korn inequality. That gives a rigid displacement $r(y)=\Ga\land y+\Gb$, $\Ga,\; \Gb\in \R^3$ such that
$$\|\zeta \cdot +r+\widehat{w}\|_{H^1(\Y^*)}\leq C\|\zeta+ e_y(\widehat{w})\|_{L^2(\Y^*)}.$$ Comparing the traces of the displacement $y\longmapsto \zeta\,y +r(y)+\widehat{w}(y)$ on the opposite faces of $\Y$ yield
$$|\zeta|\leq C'\|\zeta \cdot +r+\widehat{w}\|_{H^1(\Y^*)}$$ and then $\|e_y(\widehat{w})\|_{L^2(\Y^*)}\leq C^{''}\big(|\zeta|+\|e_y(\widehat{w})\|_{L^2(\Y^*)}\big)$. The constants only depend on $\Y^*$. That proves the inequality in the left hand-side of \eqref{EQ69}.\\
The estimate of lemma is an immediate consequence of \eqref{EQ69}.
\end{proof}
\begin{cor}\label{cor66} Let $(v,\widehat{v})$ be in $H^1(\O)^3\times L^2(\O; H^1_{per}(\Y^*))^3$  satisfying $e(v)+ e_y(\widehat{v})\in L^2(\O\times \Y^*)^{3\times 3}$. Then $v$ belongs to $H^2(\O)^3$ and $\widehat{v}\in L^2(\O; H^2_{per}(\Y^*))^3$.
\end{cor}
\subsection{The unfolded limit variational problem}
The solution of problem \eqref{var_stat_G} satisfies the estimates of Theorem \ref{TH71-bis}. Set (recall Definition \ref{def1})
\begin{equation}\label{Kper}
\begin{aligned}
&{\cal K}^k_{per}(\Y^*)\doteq \big \{ \phi\in H^k_{per}(\Y^*)^3\cap W(\Y^*)\; \; | \;\; [\phi_{\nu}]_{S}\leq 0\big\},\qquad k\in\{1,2\},\\
&L^2(\O; {\cal K}^k_{per}(\Y^*))\doteq \{\widehat{v}\in L^2(\O; H^k_{per}(\Y^*))^3\;|\; \widehat{v}(x,\cdot)\in {\cal K}^k_{per}(\Y^*)\; \hbox{for a.e. } x\in \O\Big\}.
\end{aligned}
\end{equation}
\begin{prop}  There exist a subsequence of $\{\e\}$, still denoted $\{\e\}$, $u$ in $H^1_\Gamma(\O)^3$ and  $\widehat{u}$ belonging to $L^2(\O; {\cal K}^2_{per}(\Y^*))$ such that
	\begin{equation}\label{WLimit}
	\begin{aligned}
	&\Tes(u_\e)\longrightarrow u\quad \hbox{strongly in }\;L^2(\O ; H^1(\Y^*))^3,\\
	&\Tes(\nabla u_\e)\rightharpoonup\nabla u+\nabla_y \widehat{u}\quad \hbox{weakly in }\;L^2(\O\times\Y^*)^{3\times 3},\\
	&\Tes\big(e(u_\e)\big)\longrightarrow e(u)+e_y(\widehat{u})\quad \hbox{strongly in }\;L^2(\O' \times\Y^*)^{3\times 3},\\
	&\e \Tes(\nabla e(u_\e))\rightharpoonup \nabla_{y}e_y(\widehat{u})\quad \hbox{weakly in }\;L^2(\O\times\Y^*)^{27},\\
	&{1\over \e}{\cal T}^b_\e([u_\e]_{S_\e})\rightharpoonup  [\widehat{u}]_{S}\quad \hbox{weakly in }\;L^2(\O ; H^{1/2}(S))^3,\\
	&{\cal T}^b_\e(\sigma_{\nu_\e}(u_\e))\longrightarrow \Sigma_\nu\quad \hbox{strongly in }\;L^2(\O' \times S).
	\end{aligned}
	\end{equation}
	Moreover
	\begin{equation}\label{EstC10-TER}
	\Sigma_\nu=(\sigma(u)+\sigma_y(\widehat u)_{|S})\nu\cdot \nu=\sigma_\nu(u)+\sigma_{y,\nu}(\widehat u)_{|S}\leq 0\quad \hbox{a.e. in} \quad \O \times S
	\end{equation} where $\sigma(u)(x,y)=a_{ijkl}(y)\,e_{ij}(u)(x)$, $\sigma_y(\widehat u)(x,y)=a_{ijkl}(y)\,e_{ij,y}(\widehat{u})(x,y)$ for a.e. $(x,y)\in\o \times \Y^*$.
	\end{prop}
\begin{proof} Applying \cite[Theorem 2.13]{cddgz} to the sequence $\{u_\e\}_\e$ gives a subsequence of $\{\e\}$, $u$ in $H^1_\Gamma(\O)^3$ and  $\widehat{u}$ in $ L^2(\O;H^2_{per}(\Y^*))^3$ such that convergences \eqref{WLimit}$_{1,2,4}$ hold.\\
Now, for every open set $\o$ with Lipschitz boundary satisfying $\o\Subset \O'\Subset \O$, due to estimates \eqref{EQ68} and the compact embedding theorems, one obtains
$$
\begin{aligned}
&\Qed\big(e(u_\e)\big)\rightharpoonup e(u)+e_y(\widehat{u})\quad \hbox{weakly in }\;H^1(\o; H^1(Y^*))^{3\times 3},\\
&\Qed\big(e(u_\e)\big)\longrightarrow e(u)+e_y(\widehat{u})\quad \hbox{strongly in }\;L^2(\o\times Y^*)^{3\times 3},\\
&\Tes\big(e(u_\e)\big)\longrightarrow e(u)+e_y(\widehat{u})\quad \hbox{strongly in }\;L^2(\o\times Y^*)^{3\times 3}.
\end{aligned}
$$
Thus, convergence \eqref{WLimit}$_{3}$  holds. Moreover, one has
$$
\begin{aligned}
&\Qed\big(e(u_\e)\big)_{|{\cal S}^\pm}\rightharpoonup e(u)+e_y(\widehat{u})\quad \hbox{weakly in }\;H^1(\o; H^{1/2}(\partial {\cal S}^\pm))^{3\times 3},\\
&\Qed\big(e(u_\e)\big)_{|{\cal S}^\pm}\longrightarrow e(u)+e_y(\widehat{u})\quad \hbox{strongly in }\; L^2(\o \times \partial {\cal S}^\pm)^{3\times 3}.
\end{aligned}
$$
Estimate \eqref{EQ912} implies
$$
\Teb\big(e(u_\e)\big)_{|{\cal S}^\pm}\longrightarrow e(u)+e_y(\widehat{u})_{|{\cal S}^\pm}\quad \hbox{strongly in }\;L^2(\o \times {\cal S}^\pm)^{3\times 3}.
$$
Then, from the above strong convergence,  \eqref{WLimit}$_{6}$ follows. \\
Set $\widetilde{u}_\e=u_\e-r_{u_\e}$ where $r_{u_\e}$ is defined in Subsection \ref{SS4.6}. From estimates \eqref{smalldomain} in Lemma \ref{prop5.1} one has
$$\|\Tes(\widetilde{u}_\e)\|_{L^2(\O; H^1(\Y^*))}\leq C\e\|f\|_{L^2(\O)}.$$ Then, up to a subsequence, there exists $\widetilde{u}\in L^2(\O; H^1(\Y^*))^{3}$ such that
$${1\over \e}\Tes(\widetilde{u}_\e)\rightharpoonup \widetilde{u}\quad \hbox{weakly in }\;L^2(\O; H^1(\Y^*))^{3}.$$
Hence
$$
\begin{aligned}
	&{1\over \e}{\cal T}^b_\e([\widetilde{u}_\e]_{S_\e})\rightharpoonup  [\widetilde{u}]_{S}\quad \hbox{weakly in }\;L^2(\O; H^{1/2}(S))^3,\\
	&{1\over \e}e_y\big(\Tes(\widetilde{u}_\e)\big)\rightharpoonup e_y(\widetilde{u})\quad \hbox{weakly in }\;L^2(\O; H^1(\Y^*))^{3\times 3}.
	\end{aligned}
	$$ Besides, since $e(\widetilde{u}_\e)=e(u_\e)$ in every cell $\e(\xi+\Y^*)$, $\xi\in \Xi_\e$, that gives
$\ds {1\over \e}e_y\big(\Tes(\widetilde{u}_\e)\big)=\Tes(e(\widetilde{u}_\e))$. Passing to the limit gives $e_y(\widetilde{u})=e(u)+e_y(\widehat{u})$ and then there exists $c\in L^2(\O)^3$ such that
$$\widetilde{u}(x,y)=c(x)+\nabla u(x)\, y+\widehat{u}(x,y)\quad \hbox{a.e. in }\O\times \Y^*.$$ Finally, since $[\widetilde{u}_\e]_{S_\e}=[u_\e]_{S_\e}$ convergence \eqref{WLimit}$_{5}$ is proved.
\end{proof}
\begin{rem} As a consequence of the above proposition, for every open set $\o\Subset \O$, one has $e(u)+e_y(\widehat{u})$ belongs to $H^1(\o; H^1_{per}(\Y^*))^{3\times 3}$, which  implies for $i\in\{1,2,3\}$
$$e\Big({\partial u\over \partial x_i}\Big)+e_y\Big({\partial \widehat{u}\over \partial x_i}\Big)\in L^2(\o; H^1_{per}(\Y^*))^{3\times 3}.$$
Then Corollary \ref{cor66} yields
$$u\in H^1_\Gamma(\O)^3\cap H^2_{loc}(\O)^3,\qquad \widehat{u}\in H^1_{loc}(\O; H^2_{per}(\Y^*))^3.$$
\end{rem}
\begin{thm}\label{TH107} The pair $(u,\widehat u)\in H^1_\Gamma(\O)^3\times L^2(\O; {\cal K}^2_{per}(\Y^*))$ is a solution of the following  variational inequality problem:
	\begin{equation}\label{PB-u}
	\begin{aligned}
	&\int_{\O\times \Y^*}a_{ijkl}\big(e_{ij}(u)+e_{y,ij}(\widehat{u})\big)\,\big(e_{kl}(\Phi-u)+ \, e_{y,kl}(\widehat{\phi}-\widehat{u})\big)\,dxdy+\kappa\int_{\O\times \Y^*}\nabla_y e_y(\widehat{u})\,\nabla_y e_y(\widehat{\phi}-\widehat{u})\,dxdy\\
&\quad +\int_{\O \times S}\mu\,|\Sigma_\nu| \,\big(|[\widehat{\phi}]_\tau|- |[\widehat{u}]_\tau|\big)dxd\sigma_y\geq \int_\O f\cdot (\Phi-u)\, dx,\qquad \forall (\Phi, \widehat{\phi})\in H^1_\Gamma(\O)^3\times L^2(\O; {\cal K}^2_{per}(\Y^*)).
\end{aligned}
\end{equation} 
\end{thm}
\begin{proof} Consider the test function $\ds v_\e(x)=\Phi(x)+\e \widehat{\phi}\Big(x, {x\over \e}\Big)$, for all $x\in \O^*_\e$,  where $\Phi\in H^1_\Gamma(\O)\cap {\cal C}^\infty(\overline{\O})$  and $\widehat{\phi} \in L^2(\O; {\cal K}^2_{per}(\Y^*))\cap {\cal D}(\O; H^2_{per}(\Y^*))$. The test displacement  $v_\e$ belongs to ${\cal K}_\e$. The following strong convergences hold:
	\begin{equation}\label{EQ920}
	\begin{aligned}
	&\Tes(v_\e)\longrightarrow \Phi\quad \hbox{strongly in }\;L^2(\O ; H^1(\Y^*))^3,\\
	&\Tes\big(e(v_\e)\big)\longrightarrow e(\Phi)+e_y(\widehat{\phi})\quad \hbox{strongly in }\;L^2(\O\times\Y^*)^{3\times 3},\\
	&\e \Tes(\nabla e(v_\e))\longrightarrow \nabla_{y}e_y(\widehat{\phi})\quad \hbox{strongly in }\;L^2(\O\times\Y^*)^{27},\\
	&{1\over \e}{\cal T}^b_\e([v_\e]_{S_\e})\longrightarrow  [\widehat{\phi}]_{S}\quad \hbox{strongly in }\;L^2(\O ; H^{1/2}(S))^3,\\
	&{\cal T}^b_\e(\sigma_{\nu}(v_\e))\longrightarrow \sigma_{\nu}(\Phi)+\sigma_{\nu,y}(\widehat{\phi})\quad \hbox{strongly in }\;L^2(\O \times S).
	\end{aligned}
	\end{equation}

\noindent Problem \eqref{var_Inq} also reads
$$
\begin{aligned}
&{\mathbf a}^{\e}(u_\e,u_\e)+ \kappa\e^2[\,\nabla e(u_\e)\,, \,\nabla e(u_\e)\,]_\e+ \big(\mu\, |\sigma_{\nu_\e}(u_\e)|\,,\,|[(u_\e)_{\tau_\e}]_{S_\e}|\big)_\e - \left( f, u_\e\right) \\
\leq&{\mathbf a}^{\e}(u_\e, v_\e)+ \kappa\e^2[\,\nabla e(u_\e)\,, \,\nabla e(v_\e)\,]_\e+ \big(\mu\, |\sigma_{\nu_\e}(u_\e)|\,,\,|[(v_\e)_{\tau_\e}]_{S_\e}|\big)_\e - \left( f, v_\e\right)
\end{aligned}
$$
Using the properties of the unfolding operators and convergences  \eqref{WLimit}$_{2,6}$-\eqref{EQ920}$_{2,5}$, we obtain
		\begin{equation*}
		\begin{aligned}
		&\lim_{\e\to 0}  \,{\mathbf a}^{\e}(u_{\e}, v_{\e}) 
		= \int\limits _{\O\times Y^{*}}  a_{ijkl}  \big(e_{ij}(u)+ e_{y,ij}(\widehat u)\big) \big(e_{kl}(\Phi) +  e_{y,kl}(\widehat{\phi})\big) \, dx\,dy,\\
		&\lim_{\e\to 0}  \,\int_{S_\e}\mu\, |\sigma_{\nu_\e}(u_\e)|\, |[(v_\e)_{\tau_\e}]_{S_\e}|d\sigma_\e=\int_{\O\times S}\mu\,|\Sigma_\nu| \, |[\widehat{\phi}]_\tau|dxd\sigma_y. 
		\end{aligned} 
		\end{equation*}
		Further, due to the lower semi-continuity with respect to weak topology and convergences \eqref{WLimit}$_{3,4}$ for the unfolded sequences, we obtain 
		\begin{multline*}
		\liminf_{\e\to 0}\big({\mathbf a}^{\e}(u_\e,u_\e)+ \kappa\e^2[\,\nabla e(u_\e)\,, \,\nabla e(u_\e)\,]_\e\big)\\
		\geq \int_{\O\times \Y^*}a_{ijkl}\big(e_{ij}(u)+e_{y,ij}(\widehat{u})\big)\,\big(e_{kl}(u)+e_{y,kl}(\widehat{u})\big)\,dxdy+\kappa \|\nabla_y e_y(\widehat{u})\|^2_{L^2(\O\times \Y^*)}
		\end{multline*}
	 while \eqref{WLimit}$_1$ yields $\ds \lim_{\e\to 0} (f,u_\e)=\int_\O f\cdot u\, dx$. \\
One  has 
\begin{multline*}
\int_{S_\e}\mu\, |\sigma_{\nu_\e}(u_\e)|\, |[(u_\e)_{\tau_\e}]_{S_\e}|d\sigma_\e=\e\int_{S_\e}\mu\,{\bf 1}_{\O'} \,|\sigma_{\nu_\e}(u_\e)|\, {1\over \e}|[(u_\e)_{\tau_\e}]_{S_\e}|d\sigma_\e\\
=\int_{\O\times S}\Teb(\mu)\Teb({\bf 1}_{\O'})|\Teb(\sigma_{\nu_\e}(u_\e))|\, {1\over \e}|\Teb([(u_\e)_{\tau_\e}]_{S_\e})|dx d\sigma_y.
\end{multline*}
From convergences \eqref{WLimit}$_{5,6}$, one obtains
$$
\int_{\O\times S}\mu\,{\bf 1}_{\O'}|\Sigma_\nu| \, |[\widehat{u}]_\tau|\, dxd\sigma_y =\lim_{\e\to 0}\int_{\O\times S}\Teb(\mu)\Teb({\bf 1}_{\O'})|\Teb(\sigma_{\nu_\e}(u_\e))|\, {1\over \e}|\Teb([(u_\e)_{\tau_\e}]_{S_\e})|dx d\sigma_y.
$$
Summarizing  the above convergences, that leads to
$$
\begin{aligned}
&\int_{\O\times \Y^*}a_{ijkl}\big(e_{ij}(u)+e_{y,ij}(\widehat{u})\big)\,\big(e_{kl}(u)+ \, e_{y,kl}(\widehat{u})\big)\,dxdy+\kappa \int_{\O\times \Y^*}|\nabla_y e_y(\widehat{u})|^2dxdy\\
&\qquad +\int_{\O \times S}\mu\,|\Sigma_\nu| \, |[\widehat{u}]_\tau|\, dxd\sigma_y-\int_\O f\cdot u\, dx\\
\leq &\int_{\O\times \Y^*}a_{ijkl}\big(e_{ij}(u)+e_{y,ij}(\widehat{u})\big)\,\big(e_{kl}(\Phi)+ \, e_{y,kl}(\widehat{\phi})\big)\,dxdy+\kappa\int_{\O\times \Y^*}\nabla_y e_y(\widehat{u})\,\nabla_y e_y(\widehat{\phi})\,dxdy\\
&\qquad +\int_{\O \times S}\mu\,|\Sigma_\nu| \, |[\widehat{\phi}]_\tau|dxd\sigma_y-\int_\O f\cdot \Phi\, dx.
\end{aligned}
$$ A density argument allows to conclude for every test fields in $H^1_\Gamma(\O)^3\times L^2(\O; {\cal K}^2_{per}(\Y^*))$. 
\end{proof}	

	Below, we prove the  uniqueness  of the solution of the unfolded problem  \eqref{PB-u}.
\begin{prop}\label{prop108}
The  problem \eqref{PB-u} admits a unique solution.\footnote{Proceeding as  in Section \ref{S8} with first a constant non-negative friction $\widehat{G}$ belonging to $L^2(\O\times S)$ and then  again using the Banach fixed-point theorem, one can prove the existence and the uniqueness of the solution to problem \eqref{PB-u}.    }
\end{prop}

\begin{proof}  
 Denote
$$\GN(v,\widehat{v})=\sqrt{\big(\|e(v)+e_y(\widehat{v})\|_{L^2(\O\times \Y^*)}^2 + \|\,\nabla_y e_y(\widehat{v})\,\,\|^2_{L^2(\O\times \Y^*)}},\quad \forall (v,\widehat{v})\in H^1_\Gamma(\O)\times L^2(\O; H^2_{per}(\Y^*)^3\cap W(\Y^*)).$$	Due to Lemma \ref{lem65}, $\GN$ is a norm over $H^1_\Gamma(\O)\times L^2(\O; H^2_{per}(\Y^*)^3\cap W(\Y^*))$.\\
First,  note that the Theorem \ref{TH107} gives a pair $(u,\widehat{u})$ which satisfies  the  problem \eqref{PB-u}.\\[0.8mm]
 Below we only detail the proof of the uniqueness.\\[0.8mm]
Let $(u',\widehat{u}')$ be another solution of this problem. First, choose as test fields $(\Phi,\widehat{\phi})=(u',\widehat{u}')$ in \eqref{PB-u}, then since $(u',\widehat{u}')$ is also a solution, in the corresponding problem chose as test fields $(\Phi,\widehat{\phi})=(u,\widehat{u})$. Finally, add both inequalities. That gives
	\begin{equation*}
	\begin{aligned}
	&\int_{\O\times \Y^*}a_{ijkl}\big(e_{ij}(u-u')+e_{y,ij}(\widehat{u}-\widehat{u}')\big)\,\big(e_{kl}(u-u')+ \, e_{y,kl}(\widehat{u}-\widehat{u}')\big)\,dxdy\\
	+&\kappa \int_{\O'\times \Y^*}|\nabla_y e_y(\widehat{u}-\widehat{u}')|^2dxdy \leq \int_{\O\times S}-\mu(\sigma_\nu(u-u')+\sigma_{y,\nu}(\widehat{u}-\widehat{u}'))\, \big(|[\widehat{u}]_\tau|- |[\widehat{u}']_\tau|\big)dxd\sigma_y.
\end{aligned}
\end{equation*} 
Hence
\begin{multline*}
\overline{\alpha}\|e(u-u')+e_y(\widehat{u}-\widehat{u}')\|^2_{L^2(\O\times \Y^*)}+\kappa\|\nabla_y e_y(\widehat{u}-\widehat{u}')\|^2_{L^2(\O\times \Y^*)}\\
\leq \|\mu\|_{L^\infty(\O)}\|\sigma_\nu(u-u')+\sigma_{y,\nu}(\widehat{u}-\widehat{u}'))\|_{L^2(\O\times S)}\|[\widehat{u}]_\tau-[\widehat{u}']_\tau\|_{L^2(\O\times S)}.
\end{multline*}
The above inequality, \eqref{JumpS-bis} and \eqref{EQ75} (applied with the displacement $y\longrightarrow e(u-u')(x) y+(\widehat{u}-\widehat{u}')(x,y)$ defined for a.e. $x\in \O$)  lead to
$$\min\{\overline{\alpha},\kappa\}\big(\GN(u-u', \widehat{u}-\widehat{u}')\big)^2\leq \|\mu\|_{L^\infty(\O)}C_1\|\Ga\|_{W^{1,\infty}(S)}\GN(u-u',\widehat{u}-\widehat{u}') C_0\|e_y(\widehat{u}-\widehat{u}')\|_{L^2(\O\times \Y^*)}.$$
Condition \eqref{Contraction} gives the uniqueness of the solution.
\end{proof}
\noindent Due to the nonlinearity of the terms involving the tangential jumps,  a homogenized problem can not be obtained.	
\section{Annex}\label{S11}
In this section we denote $u_{\e,G,\kappa}$ the solution of \eqref{var_stat_G} (resp. $U_{\e,G}$ the solution of \eqref{var_stat_G-0}) with $G\in {\cal C}^{**}_\e$. 
\begin{prop} \label{prop111} There exists  $(u_G,\widehat{u}_G)\in  H^1_\Gamma(\O)^3\times L^2(\O; {\cal K}^1_{per}(\Y^*))$\footnote{see \eqref{Kper} for the definition of this space} such that when $(\e,\kappa)$ goes to $(0,0)$ (resp. $\e$ goes to $0$)
\begin{equation}\label{EQFIN-0}
\begin{aligned}
&\Tes(u_{\e,G,\kappa})\longrightarrow u_G\quad \hbox{strongly in }\;L^2(\O; H^1(\Y^*))^3,\\
&\Tes(\nabla u_{\e,G,\kappa})\longrightarrow  \nabla u_G+\nabla_y \widehat{u}_G\quad \hbox{strongly in }\;L^2(\O\times\Y^*)^{3\times 3},\\
\hbox{(resp.}\;\; 
&\Tes(U_{\e,G})\longrightarrow u_G\quad \hbox{strongly in }\;L^2(\O; H^1(\Y^*))^3,\\
&\Tes(\nabla U_{\e,G})\longrightarrow  \nabla u_G+\nabla_y \widehat{u}_G\quad \hbox{strongly in }\;L^2(\O\times\Y^*)^{3\times 3}.\hbox{)}
\end{aligned}
\end{equation}
The couple $(u_G,\widehat{u}_G)$ is the unique solution of the following variational inequality:
\begin{equation}\label{EQFIN}
\begin{aligned}
&\int_{\O\times \Y^*}a_{ijkl}\big(e_{ij}(u_G)+e_{y,ij}(\widehat{u}_{G})\big)\,\big(e_{kl}(\Phi-u_G)+ \, e_{y,kl}(\widehat{\phi}-\widehat{u}_{G})\big)\,dxdy\\
&\quad +\int_{\O\times S} G \,\big(|[\widehat{\phi}]_\tau|- |[\widehat{u}_{G}]_\tau|\big)dxd\sigma_y\geq \int_\O f\cdot (\Phi-u_G)\, dx,\\
&\quad \forall (\Phi, \widehat{\phi})\in H^1_\Gamma(\O)^3\times L^2(\O; {\cal K}^1_{per}(\Y^*)).
\end{aligned}
\end{equation}
\end{prop}
\begin{proof} First, recall that  the estimates \eqref{EstC1-ter}  hold. So, there exist a subsequence of $\{\e,\kappa\}$, still denoted $\{\e,\kappa\}$, and $(u_G,\widehat{u}_G)\in  H^1_\Gamma(\O)^3\times L^2(\O; {\cal K}^1_{per}(\Y^*))$ such that  
$$
\begin{aligned}
&\Tes(u_{\e,G,\kappa})\longrightarrow u_G\quad \hbox{strongly in }\;L^2(\O; H^1(\Y^*))^3,\\
&\Tes(\nabla u_{\e,G,\kappa})\rightharpoonup\nabla u_G+\nabla_y \widehat{u}_G\quad \hbox{weakly in }\;L^2(\O\times\Y^*)^{3\times 3},\\
&\sqrt\kappa \e\Tes(\nabla e(u_{\e,G,\kappa}))\rightharpoonup 0\quad \hbox{weakly in }\;L^2(\O\times\Y^*)^{27},\\
&{1\over \e}{\cal T}^b_\e([u_{\e,G,\kappa}]_{S_\e})\rightharpoonup  [\widehat{u}_G]_{S}\quad \hbox{weakly in }\;L^2(\O ; H^{1/2}(S))^3.
\end{aligned}
$$ Due to the lower semi-continuity with respect to the weak topology and the above convergences  for the unfolded sequences, we obtain 
\begin{multline*}
\liminf_{(\e,\kappa)\to (0,0)}\Big({\mathbf a}^{\e}(u_{\e,G,\kappa},u_{\e,G,\kappa})+ \kappa\e^2\|\nabla e(u_{\e,G,\kappa})\|^2_{L^2(\O^*_\e)}+{1\over \e}\int_{\O\times S}\Teb(G)\, \Teb(|[u_{\e,G,\kappa}]_\tau|)dxd\sigma_y-\int_{\O^*_\e} f\cdot u_{\e,G,\kappa}\, dx\Big)\\
		\geq \int_{\O\times \Y^*}a_{ijkl}\big(e_{ij}(u_{G})+e_{y,ij}(\widehat{u}_{G})\big)\,\big(e_{kl}(u_{G})+e_{y,kl}(\widehat{u}_{G})\big)\,dxdy+\int_{\O\times S} G \, |[\widehat{u}_{G}]_\tau|\,dxd\sigma_y-\int_{\O} f\cdot u_{G}\, dx
\end{multline*}
Now,  consider the test displacement $v_\e\in{\cal K}_\e$ introduced in  the proof of Theorem \ref{TH107}. One has
$$
\begin{aligned}
&\lim_{(\e,\kappa)\to (0,0)}\big({\mathbf a}^{\e}(u_{\e,G,\kappa}, v_\e)+ \kappa\e^2[\,\nabla e(u_{\e,G,\kappa})\,, \,\nabla e(v_\e)\,]_\e+ \big(G,|[(v_\e)_{\tau_\e}]_{S_\e}|\big)_\e - \left( f, v_\e\right)\big)\\
=&\int_{\O\times \Y^*}a_{ijkl}\big(e_{ij}(u_{G})+e_{y,ij}(\widehat{u}_{G})\big)\,\big(e_{kl}(\Phi)+e_{y,kl}(\widehat{\phi})\big)\,dxdy+\int_{\O\times S} G \, |[\widehat{\phi}]_\tau|\,dxd\sigma_y-\int_{\O} f\cdot \Phi\, dx
\end{aligned}
$$
Finally, a density argument allows to obtain inequality \eqref{EQFIN}. Since the problem \eqref{EQFIN} admits a unique solution,  the whole sequences converge to their limits. As in \cite{GMO}, we prove that the convergences \eqref{EQFIN-0}$_{2,4}$ are strong convergences.\\
Proceeding in the same way with the sequence $\{U_{\e,G}\}_\e$,  one obtains the same limit problem.
\end{proof} 
The homogenization of problem \eqref{var_stat} (and also of problem \eqref{var_Inq} for small $\kappa$) remains  an open problem.

\end{document}